\documentclass{amsart}
\usepackage{cmap}
\usepackage{amssymb}
\usepackage{color}
\usepackage{graphicx}
\usepackage{amsmath}
\usepackage{hyperref}
\usepackage{fullpage}
\hypersetup{
	colorlinks = true, 
	urlcolor = blue, 
	linkcolor = red, 
	citecolor = blue 
}
\usepackage{enumerate}

\theoremstyle{plain}
\newtheorem{thm}{Theorem}[section]
\newtheorem{lemma}[thm]{Lemma}
\newtheorem{prop}[thm]{Proposition}
\newtheorem{cor}[thm]{Corollary}

\newtheorem{remark}{Remark}
\newtheorem{claim}[thm]{Claim}
\newtheorem*{thm*}{Theorem}
\theoremstyle{definition}
\newtheorem{defn}[thm]{Definition}

\def \m{\bar}

\newcommand{\A}{\mathcal{A}}

\newcommand{\N}{\mathbb{N}}

\newcommand{\zd}{\mathbb{Z}^d}

\renewcommand{\subset}{\subseteq}

\newcommand{\Z}{\mathbb{Z}}
\newcommand{\parity}{\mathop{\mathrm{parity}}}
\newcommand{\dom}{\mathop{\mathrm{dom}}}

\def\hom{\textit{Hom}}

\def\free{\textit{Free}}

\title{Borel factors and embeddings of systems in subshifts}

\author{
	Nishant Chandgotia
	\and
	Spencer Unger
}

\address{Tata Institute of Fundamental Research, Bengaluru, India}
\email{nishant.chandgotia@gmail.com}
\address{University of Toronto, Mississauga, Canada}
\email{unger.the.aronszajn.trees@gmail.com}
\subjclass[2020]{Primary: 37B10, 54H05, Secondary: 05C15}
\keywords{Borel dynamical systems, shift spaces, vertex colourings, edge colourings, tilings, bi-infinite directed Hamiltonian paths, entropy, embeddings}

\begin{document}
	\begin{abstract} In this paper we study the combinatorics of free Borel
actions of the group $\Z^d$ on Polish spaces. Building upon recent work by
Chandgotia and Meyerovitch,  we introduce property $F$ on $\Z^d$-shift spaces
$X$ under which there is an equivariant map from any free Borel action to the
free part of $X$. Under further entropic assumptions, we prove that  any
subshift $Y$ (modulo the periodic points) can be Borel embedded into $X$.
Several examples satisfy property $F$ including, but not limited to, the space
of proper $3$-colourings, tilings by rectangles (under a natural arithmetic
condition), proper $2d$-edge colourings of $\Z^d$ and the space of bi-infinite
Hamiltonian paths. This answers questions raised by Seward, and Gao-Jackson,
and recovers a result by Weilacher and some results announced by
Gao-Jackson-Krohne-Seward.  \end{abstract} 
\maketitle

\section{Introduction}\label{section: Introduction}

In this paper we focus on the following question: Under what hypothesis can
subshifts model all free $\Z^d$ dynamical systems? While there are numerous
contexts in which this question can be placed, we will work with free Borel
$\Z^d$ systems on Polish spaces. Our inspiration comes from two recent groundbreaking results by Hochman:
In his first paper \cite{MR3077948}, Hochman showed that barring a
\emph{universally null set}, that is, sets which have measure zero for all
invariant probability measures, any Borel $\Z$ system (with small enough
entropy) can be equivariantly embedded into mixing shifts of finite type. It
remained open to extend the embedding to  ``dark matter'', a term coined by Mike
Boyle, referring to part of the space which does not have any invariant
probability measures. For many interesting examples ``dark matter'' can
constitute the entire space. This was addressed in a later paper by Hochman
\cite{MR3880210} where he showed that that such an embedding exists even without
leaving out a universally null set. Recently, in a private communication, Seward
and Hochman mentioned that they can construct an embedding of this ``dark matter'' into the full shift for
countable group actions.

Fix $d\geq 2$. Our study begins with a recent paper \cite{MR4311117}
where the authors defined a property called \emph{flexibility}: If a dynamical
system is flexible, the authors proved that free $\Z^d$ systems can be
embedded in it barring a universally null set under natural entropic
restrictions. Numerous dynamical systems like the space of proper $3$-colourings
and domino tilings (and many non symbolic examples) possess this property. In
this paper we consider a similar condition on subshifts called property $F$ (Definition \ref{definition: Property F}) which helps us prove similar results in the
context of Borel dynamics. For systems with property $F$ we showed the following
results.
\begin{enumerate} 
\item All free Borel $\Z^d$ systems admit equivariant Borel maps to the free
part of a system with property $F$. (Theorem \ref{theorem: property F factor})
\item All subshifts (under suitable restrictions on the
entropy) can be Borel embedded in a subshift with property $F$. (Theorem
\ref{theorem: property F embedding subshift}) 
\end{enumerate}

In this paper, we prove that many subshifts have property $F$:

\begin{enumerate}
	\item Hom-shifts: The space of graph homomorphisms from $\Z^d$ to a fixed undirected graph with an odd cycle. This covers many interesting examples like the space of proper $3$-colourings. (Section \ref{Section: Hom-shifts}).

        Gao and Jackson raised the following question in
\cite{gao2015countable}: What is the Borel chromatic number of the free part of
the full shift? Recently it has been announced in \cite{gao2018continuous} that
Borel chromatic number is less than equal to $3$. Our results confirm this
announcement. In contrast it has been shown in \cite{gao2018continuous} that the
continuous chromatic number of the free part of the full shift is $4$.
	
        \item Rectangular tiling shifts: Tilings of $\Z^d$ by boxes which
satisfy a natural necessary and sufficient coprimality condition. This covers
examples like that of domino tilings (also known as perfect matchings). (Section
\ref{Section: Tiling shift})
	
        It was a question by Gao and Jackson in \cite{gao2015countable} and
through personal communication by Brandon Seward whether this condition was
sufficient. The resolution of this question in the case of domino tilings for
$\Z^2$ actions was mentioned to us over personal communication by Brandon Seward
and is going to be part of their upcoming publication \cite{gaotoappear}. In
contrast it has been shown that the free part of the full shift does not have
any continuous perfect matching \cite{gao2018continuous}.

\item The space of
directed bi-infinite Hamiltonian paths on $\Z^d$. (Section \ref{Section:
Directed Hamiltonian})
	
    This is an important example which does not precisely fall in the category
of subshifts (since it is not compact). Yet our techniques extend to this with
small modifications. It recovers, partially, a weak version of Dye's theorem
\cite{dye1959groups} which states that any free $\Z^d$ action is orbit
equivalent to a $\Z$-action (barring a null set of an invariant ergodic
probability measure). Our result proves the existence of such an orbit
equivalence where a null set need not be removed. In addition the orbit
equivalence maps the generators of the $\Z$ action to generators of the $\Z^d$
action. This result for $\Z^2$ actions was also mentioned to us by Brandon
Seward over personal communication and that it will be part of their upcoming
publication \cite{gaotoappear}. A similar orbit equivalence was obtained in
\cite{downarowicz2021multiorders} for $\Z^2$ dynamical systems barring a set of
measure zero for a given free ergodic probability measure.

\item The space of colourings of edges of $\Z^d$ for $d\geq 2$ by $k$
colours for $k\geq 2d$. (Section \ref{section: proper edge colouringe})

It was asked in \cite{gao2018continuous} to determine the Borel edge-chromatic
number of $\Z^2$ actions. Our theorem shows that the answer is $2d$ for $\Z^d$
actions. In contrast \cite{gao2018continuous} shows that the continuous
chromatic number of the free part of the full shift on $\Z^2$ is $5$. This
recovers results by Weilacher in \cite[Theorem 2]{weilacher2021borel}.  Our
result on edge colorings is also a consequence of results in
\cite{bencs2021factor}.

\end{enumerate}

\subsection{Selected literature and general context} To put our results in
context, we describe several closely related areas.

The `original' context for such questions comes from Krieger's generator theorem
\cite{MR0422576}. It states that under natural assumptions on the entropy any
free ergodic probability preserving $\Z$ system $(X,\mu, T)$ can be embedded
into mixing shifts of finite type. In other words mixing shifts of finite type
are `universal'. Naturally the question arose, which other systems are
`universal'? It has been found time and again that at the heart of the matter
are certain topological mixing properties: Quas and Soo proved universality under `specification' and some technical
conditions \cite{MR3453367}. These technical conditions were loosened by Weiss
and were finally gotten rid of by Burguet \cite{burguet2020topological}.
`Universality' under similar mixing conditions can also be found for $\Z^d$
symbolic systems where $d>2$, see \c{S}ahin and Robinson
\cite{MR1844076} and Pavlov \cite{MR3162822}. In
\cite{MR4311117}, Chandgotia and Meyerovitch considered very general
conditions which imply ``almost Borel universality'' for various systems which
don't satisfy these mixing conditions but still have sufficiently many orbit
segments which can be `glued' to one another. Further the embeddings there are
modulo a universally null set, a set with measure zero with respect to all
invariant probability measures, as opposed to modulo a set of measure zero with
respect to a specific one. One is immediately led to the question: When can we
embed a space which has no invariant probability measure
\cite{weiss1989countable}? In \cite{weiss1989countable}, Weiss had shown that
such spaces have countable generators (later generalised to countable groups in
\cite{jackson2002countable}). Tserunyan showed that all such countable groups
actions have a $32$ set generator under the assumption of local compactness
\cite{tserunyan2015finite} and for $\Z$ actions Hochman showed that these spaces can be embedded
into any mixing shift of finite type \cite{MR3880210}. This question is rather
delicate and sets the context for our paper: In our results we do not remove
sets which admit no invariant probability measures.  While we do not prove the
full universality of these systems we make considerable progress towards it.


An important inspiration for our work also comes from the systematic study of
Borel graphs begun by Kechris, Solecki and Todor\v{c}evi\'c  \cite{MR1667145}.
Identify $\Z^d$ with its Cayley graph with standard generators. A free $\Z^d$
action on a Polish space $X$ induces a natural graph structure on $X$ which can
be thought as a measurable collection of copies of $\Z^d$. To understand what
this graph is like, we can now try to understand its Borel chromatic number, how can
it be tiled and other similar properties. We refer the reader to some
comprehensive surveys by Kechris and Kechris-Marks for more references and
details \cite{kechris2019theory,kechris2016descriptive}. Kechris, Solecki and
Todor\v{c}evi\'c \cite[Sections 4.8, 4.9]{MR1667145} proved that the Borel chromatic number and the Borel edge chromatic number of the
$\Z^d$ full shift is $2d+1$ and $4d-1$ respectively and asked what is the
Borel chromatic number is if periodic points were removed.

This was taken up by Gao and Jackson in \cite{gao2015countable} where they
proved many results about the free part of the full shift on $\Z^2$ in both the
Borel and the continuous context. Modulo the periodic points, they proved that
the continuous chromatic number is less than and equal to $4$ and that it can be
continuously tiled by rectangles of the type $n\times n$, $n\times (n+1)$,
$(n+1)\times n$ and $(n+1)\times (n+1)$.  Moreover they showed that these
results are sharp in the sense that if one were to remove one of the rectangles
or reduce the number of colours, then such continuous maps no longer exist.
Naturally we are lead to the question of the existence of a Borel colouring or
tiling. In \cite{gao2018continuous}, Gao, Jackson, Krohne and Seward studied
these problems very deeply in the continuous context and announced that the free
part of the full shift has Borel chromatic number $3$, can be perfectly matched
and asked about the edge chromatic number. Modulo a null set for a given ergodic
measure some of these questions were answered by \c{S}ahin \cite{MR2573000},
Prikhod'ko \cite{MR1716239}, by \c{Sahin} and Robinson \cite{MR2050041} and by
Chandgotia and Meyerovitch \cite{MR4311117}.  Our results
characterise when a given free Borel action of $\Z^d$ can be Borel tiled by a
given set of rectangles, prove that the Borel edge chromatic number is $2d$ and
recovers that the Borel chromatic number is at most $3$. The result regarding
the edge chromatic number can be found in \cite[Theorem 2]{weilacher2021borel}
and is also a consequence of \cite{bencs2021factor}. Proof sketches for some of
these results can be found in \cite{grebik2021local}.

The motivation for Gao and Jackson came from a well-known question by Weiss
\cite{MR737417} about hyperfiniteness for actions of countable
amenable groups. A Borel action of a countable group on a Polish space gives
rise to the orbit equivalence relation which is a countable Borel equivalence
relation. Such a relation is called \emph{hyperfinite} if it is an increasing
union of finite Borel equivalence relations. By versions of Rokhlin's lemma
\cite{rohlin1948general,MR0014222,MR0316680, ornstein1987entropy}, it can be
shown that the equivalence relations arising from the action of amenable groups
on Polish spaces are hyperfinite modulo a universally null set.  Weiss asked
whether removing the universally null set is necessary. This was answered
positively for
$\Z^d$ actions by Weiss, for finitely-generated groups of polynomial growth by
Jackson, Kechris, and Louveau \cite{jackson2002countable}, for abelian groups by
Gao and Jackson \cite{gao2015countable} and for locally Nilpotent groups by
Schneider and Seward \cite{schneider2013locally}. Recently this has been
generalised to a large class of groups by Conley, Jackson, Marks, Seward and
Tucker-Drob \cite{conley2020borel}. Given that an action is hyperfinite, one can ask
whether the witness to hyperfiniteness, that is, the finite Borel equivalence
relations approximating them, are `nice' and whether they `respect' the geometry
of the acting group. This can be seen as a generalisation of a marker lemma in
the Borel context and can be found in \cite{gao_forcing_2015} (look at
\cite[Theorem 5.5]{marks_borel_2017} for a proof). The questions about the
Borel chromatic numbers of actions and Borel tilings makes use of these
`nice' witnesses to hyperfiniteness.

We remark that the continuous version of our results are especially subtle and
lot remains to be explored. For instance we know that modulo the periodic points
the $2$-full shift and the proper $3$ colourings are Borel isomorphic to one
another \cite{MR3880210,MR3077948} but these maps can't be made continuous
\cite{salo2021conjugacy}. Continuous embeddings for $\Z^2$ subshifts is known
under strong mixing conditions \cite{MR1972240,MR2085910} and there are
indications that it would fail otherwise \cite[Section
12]{MR4311117}. Results by Gao and Jackson \cite{gao2015countable}
and Gao-Jackson-Krohne-Seward \cite{gao2018continuous} further illustrate the
discrepancy between the Borel and the continuous results.

Finally, in  a recent paper Anton Bernshteyn  in
\cite{bernshteyn2020distributed} has pointed out interesting relations between
distributed computing, local algorithms and questions in descriptive
combinatorics. See also \cite{brandt2021local, grebik2021local}. 

\subsection{Organisation of the paper and a few words about the proofs}
The paper is organized as follows. In Section \ref{section: main_results} we state the main results of the paper. In Section \ref{section: preliminaries}, we
give some definitions that will be used throughout. In Section \ref{section: Rokhlin's lemma}, we give a proof of a theorem of Gao, Jackson,
Krohne and Seward using the Baire category theorem and use this to show that typical constructions done in the context of invariant probability measures
(some consequences of the Rokhlin's lemma) do not work in the Borel context.  In
Section \ref{section: toast}, we modify a different theorem of Gao, Jackson,
Krohne and Seward to suit our combinatorial constructions. This can be thought
of as a weak version of Rokhlin's lemma where we prove some additional
geometrical properties of the towers/hyperfiniteness witness. In Section
\ref{section: Tiling complements}, we prove several combinatorial results about
tilings of the hyperfiniteness witness from Section \ref{section: toast} by nice
rectangular boxes.  In Section \ref{section: propertyF}, we introduce an
extension property (Property F) for patterns in $\zd$ as well as a version of
this property with markers and apply them to produce Borel maps.  In particular,
we show that if a system $(Y,T)$ has property $F$, then there is an equivariant
Borel map from any free Borel $\Z^d$ system to $(Y,T)$.

In Section \ref{section: markers}, we introduce a marker lemma for subshifts with property $F$. This is then used to show in Section \ref{section: embedding} that if a shift space $(Y,T)$ satisfies the version of Property $F$ with markers, then every subshift with low enough entropy admits a Borel embedding into $(Y,T)$.  In the remaining sections we prove that the spaces from Theorem \ref{Theorem: Main} satisfy Property $F$ and give entropy bounds for the spaces in Theorem \ref{Theorem: Main1}: Hom-shifts  in Section \ref{Section: Hom-shifts}, rectangular tiling shifts in Section \ref{Section: Tiling shift}, directed bi-infinite Hamiltonian paths in Section \ref{Section: Directed Hamiltonian} and proper edge colourings in Section \ref{section: proper edge colouringe}.

\section*{Acknowledgments} This project began while the first author was
attending a presentation by the second author in the descriptive set theory
seminar at the Hebrew University of Jerusalem, where both the authors were
postdocs. The vibrant academic culture  at the Hebrew University was a great
inspiration to us. While a lot of this work was disrupted by the COVID-19
pandemic, the first author would like to extend special gratitude towards Shahar
Mozes, Benjamin Weiss, Zemer Kosloff and the Hebrew University of Jerusalem at
large who ensured financial stability and mental peace in difficult
circumstances. In addition, we were greatly encouraged and influenced by
numerous questions and discussions with Benjamin Weiss, Tom Meyerovitch, Brandon
Seward, Anton Bernshteyn, Andrew Marks, Omer Ben-Naria, Mike Boyle and Zemer Kosloff. The
first author was funded by ISF grants 1702/17, 1570/17 and ISF-Moked grants
2095/15 and 2919/19.  The second author was partially supported by NSF grant
DMS-1700425 and an NSERC Discovery Grant.

\section{Main results}\label{section: main_results}
\subsection{Notation}
\noindent A \emph{Borel $\zd$ system} is a pair $(X, T)$ where $X$ is a standard
Borel space and $T$ is an action of $\zd$ by Borel automorphisms.  Given a Borel
$\zd$ system $(X, T)$ we let $\free(X)=\{x\in X~:~\mbox{ for all }\gamma \in
\zd, T^{\gamma}(x)\neq x\}.$  The action is called \emph{free} if $\free(X)=X$. If $X$ is compact and $T$ is a $\Z^d$ action on $X$ by homeomorphisms, then the pair $(X, T)$ is called a $\zd$ topological dynamical system. 

An important action that we will focus on will be \emph{the full-shift}
$(\A^{\Z^d}, \sigma)$ where $\A$ is a finite set, $\sigma$ is the $\Z^d$ action
on $\A^{\zd}$ given by $(\sigma^\gamma(x))_\delta= x_{\gamma+ \delta}.$
We give the set $\A$ the discrete topology making $\A^{\Z^d}$ compact under the
product topology and the action $\sigma$ is by homeomorphisms. In this paper we
will work with subsets $X\subset \A^{\Z^d}$ which are invariant under the
shift-action; we will call them \emph{symbolic systems}. If in addition they are compact as well they will be called \emph{subshifts}. Given a set $B\subset \zd$ we denote the \emph{language} of $X$ on the set $B$ by $\mathcal{L}(X, B)=\{x|_B~:~x\in X\}$
and the \emph{language} of $X$ by $\mathcal{L}(X)=\cup_B\mathcal{L}(X, B)$ where
the union is over all finite sets $B$. A \emph{pattern} is an element of
$\mathcal{L}(\A^{\zd})$ while a \emph{configuration} is an element of
$\A^{\zd}$.

Elements of $\zd$ will typically be written as $\alpha, \beta, \gamma, \delta,
\ldots$ where for example $\gamma=(\gamma_1, \gamma_2, \ldots, \gamma_d)$.  The
standard generators of $\zd$ will be denoted by $\epsilon^i$, that is,
$\epsilon^i_j = 1$ if $i = j$ and $0$ otherwise. The other unit vectors are
denoted by $\epsilon^i; d<i \leq 2d$ where $\epsilon^i=-\epsilon^{i-d}$. 

\subsection{Examples}\label{section: Examples of F property}

Throughout the paper we will work with the undirected Cayley graph of $\Z^d$
given by the standard generators, that is, the graph with vertex set $\Z^d$ and
edges of the form $\{\gamma, \gamma \pm \epsilon^i\}$ for $i \leq d$.

\subsubsection{Hom-shifts} Let $H$ be a finite undirected graph without multiple
edges. We denote by $\hom(\zd, H)$ the collection of graph homomorphisms
(adjacency preserving maps) from $\zd$ to $H$. These are called
\emph{hom-shifts}. An important special case is proper $q$-colorings of $\zd$,
which can be written as $\hom(\zd, K_q)$ where $K_q$ is the complete graph.

\subsubsection{Rectangular tiling shifts} A \emph{box} in $\Z^d$ is a
product of intervals.  Given a finite set of boxes $\mathcal{T}$,  a
\emph{tiling} of a set $B \subset \Z^d$ by $\mathcal{T}$ is a function $a$ from
$B$ to $\mathcal{T}$ such that for all all $R \in \mathcal{T}$, $a^{-1}(R)$ is a
disjoint union of translates of $R$.  The space of tilings $X(\mathcal{T})$ is
the set of tilings of $\zd$ by $\mathcal{T}$.  We can assume that $\mathcal{T}$
consists of boxes $R$ of the form $\prod_{i \leq d}[0,n_i^R)$.

Our results will pertain to the so-called \emph{coprime
rectangular shifts} $X(\mathcal{T})$ where for all
$1\leq i \leq d$, $\gcd(n_i^R; R\in \mathcal T) = 1$.  An important rectangular
shift is that \emph{domino tilings},  $X(\mathcal{T})$ where $\mathcal{T} = \{
\{\bar{0},\epsilon^i\} \mid i \leq d \}$ where we recall that $\epsilon^i$ is
the $i^{th}$ standard generator of $\zd$.  Domino tilings are sometimes also referred
to as perfect matchings and dimer tilings.

\subsubsection{Directed bi-infinite Hamiltonian paths in $\zd$: } \cite{downarowicz2021multiorders}

Let $X_{bi}$ be the set of all functions $x:\zd \to \{\epsilon^1, \dots,
\epsilon^{2d}\}$ such that the graph on $\zd$ with directed edges of the form $(\gamma,
x_\gamma+\gamma)$ is a directed bi-infinite path.

\subsubsection{Proper edge colourings in $\Z^d$: }

For $t\in \N$ let $E_t$ be the set of injective functions from $\{\epsilon^1, \dots,
\epsilon^{2d}\}$ to the set of colours $\{1,2, \ldots, t\}$. We define \emph{the set of proper $t$-edge colourings of $\Z^d$} as
$$X^{(t)}=\{x\in (E_t)^{\Z^d}~:~ \text{ if $\epsilon$ is a unit vector and $\delta\in \Z^d$ then $x_{\delta}(\epsilon)= x_{\delta+\epsilon}(-\epsilon)$}\}.$$
Equivalently we can think of them as colouring on edges such that edges sharing a vertex are mapped to distinct colours. It is obvious that $X^{(t)}$ is non-empty if and only if $t\geq 2d$. 
\subsection{Main results}

Our main results are the following:

\begin{thm}\label{Theorem: Main}
	Let $(X,S)$ be any of the following systems.
	\begin{enumerate}
		\item $\hom(\zd,H)$ where $H$ is a graph which is not bipartite.
		\item A coprime rectangular tiling shift.
		\item The space of bi-infinite Hamiltonian paths.
		\item The set of proper $k$-edge colourings of $\Z^d$ where $k\geq 2d$ and $d>1$.
	\end{enumerate}
	For all free Borel $\zd$ actions $(Y, T)$ there exists an equivariant
Borel map $\phi: (Y, T)\to (\free(X), S)$.
\end{thm}

For a $\Z^d$ topological dynamical system $(Y, T)$ let $h_{top}(Y,T)$ denote the topological entropy of $(Y, T)$.

\begin{thm}\label{Theorem: Main1}
	Let $(X,S)$ be any of the following systems.
	\begin{enumerate}
		\item $\hom(\zd,H)$ where $H$ is a graph which is not bipartite.
		\item The space of domino tilings
		\item The set of proper $t$-edge colourings of $\Z^d$ where $t\geq 2d$ and $d>1$
	\end{enumerate}
	For all shift spaces $(Y, T)$, if $h_{top}(Y, T)<h_{top}(X, S)$ then
there exists an equivariant Borel embedding $\phi: (\free{(Y)}, T)\to (X, S)$.
\end{thm}

\begin{remark}
We note that for $d=1$, the space of proper $t$-edge colourings is equivalent to the space of proper $t$-vertex colourings (which is covered by the space of graph homomorphisms). Thus the same results for proper $t$-edge colourings holds for $t\geq 3$ when $d=1$. 
\end{remark}

\section{Preliminaries} \label{section: preliminaries}

In this section, we give some definitions which connect the dynamics of our
systems to the finite combinatorics of their Cayley graphs.  Let $(X,T)$ be a
system.  We denote by $G_T$ its natural Cayley graph structure, meaning, that
the vertex set of $G_T$ is the space $X$ and there is an edge between $x$ and
$y$ in $X$ if there is a standard generator $\gamma$ of $\mathbb Z^d$ such that
$T^{\gamma}(x)=y$.  We note that the connected components of $G_T$ are copies
of the usual Cayley graph of $\zd$ given by the standard generators.  On the
other hand, we note that Borel subsets $B$ of $X$ often meet every connected
component, so (for example) to define a Borel function on $B$ we must work
simultaneously in every connected component.

For a Borel set $B$, we will refer to the connected components of the induced
subgraph of $G_T$ with vertex set $B$ as the connected components of $B$.
Suppose that $P$ is a property of finite subsets of $\zd$ which is invariant
under translation.  Let $C\subset X$ be contained in the orbit of a point $x\in X$. We will say that $C$ has property $P$ if $\{\gamma \in \zd \mid T^{\gamma}(x) \in C \}$ has property $P$. For instance if $T$ is a free $\Z^d$ action on $X$ then $T^{[-m,m]^d}(x)$ is a cube for all $x\in X$. 

For a Borel set $B \subseteq X$, $m \in \N$ and $F\subset \zd$, we write 
$$T^F(B)=\cup_{\gamma\in F}T^{\gamma}(B)\text{ and }\partial^m B
= T^{[-m,m]^d}(B) \setminus B;$$ 
where the latter is the $m$-external boundary of $B$. We write
$\partial B$ for $\partial^1B$.  For $i \leq d$, we define $\partial_i B$ to be
the inner boundary in the direction $i$, that is, the set of $x \in B$ such that either $T^{\epsilon^i}(x) \notin B$ or $T^{-\epsilon^i}(x) \notin B$. Similarly
$\partial_i^+(B)$ is the set of $x\in B$ such that $T^{\epsilon^i}(x) \notin B$ while $\partial_i^-(B)$ is the set of $x\in B$ such that $T^{-\epsilon^i}(x) \notin B$. All these sets
are Borel.  We will use similar definitions for $B \subseteq \zd$.

\begin{defn}\label{defn: k grid union} A set $B\subset \Z^d$ is a called a\emph{
$k$-grid union} if there exists $C\subset k\Z^d$ and $\gamma \in \zd$ such that
$B=\gamma+C+[1,k]^d$ and both $B$ and $\Z^d\setminus B$ are connected.  We call $\gamma$
the \emph{offset} of $B$.  \end{defn}

\begin{defn} A finite set $C \subseteq \zd$ is \emph{coconnected} if the
induced subgraph of the Cayley graph on $\zd \setminus C$ is connected.
\end{defn}

\section{Is there a simple counterpart to Rokhlin's lemma in Borel dynamics?}\label{section: Rokhlin's lemma}

In this section, we give simple topological proofs of some theorems by Gao,
Jackson, Krohne and Seward from \cite{gao_forcing_2015} and use them to explain
why the techniques of ergodic theory are not enough for the Borel context of our
paper.

We recall Rokhlin's lemma which is the cornerstone of many constructions in
ergodic theory.

\begin{thm}[Rokhlin's lemma \cite{rohlin1948general,MR0014222} for $\mathbb Z$
actions and \cite{MR0316680} for $\mathbb Z^d$ actions]{\label{theorem: Rokhlins
lemma}} Consider a free measure preserving action of $\mathbb Z^d$ on a
probability space  $(X, \mu)$. For all $n \in \mathbb N$ and $\epsilon>0$, there
exists a Borel set $A\subset X$ such that $\mu(A)>1-\epsilon$ and the connected
components of $A$ are boxes with side length $n$.  \end{thm}

Often in ergodic theory (for instance in many proofs of the Krieger's generator
theorem \cite{MR0422576}), given a topological space $Y$,  a function $f: X\to
Y$ is constructed in the following fashion. A sequence of constants $\epsilon_n$
are chosen such that $\epsilon_n\to 0$ and sets $A_n$ are chosen by Rokhlin's
Lemma (the connected components of $A_n$ are boxes with side length $k_n$ large
enough). Now for each $n$, a function  $f_n: A_n\to Y$ is constructed and the
hope is that the sequence $f_n$ converges to a function $f$. To achieve this,
$f_{n+1}$ is constructed using $f_n$ by modifying it on a small set $B_n$ and
giving an appropriate definition to $f_{n+1}$ on $A_{n+1}\setminus A_n$. The
connected components of $B_n$ contain boxes of larger and larger sizes but they satisfy a summability condition: $\sum_{n\in \mathbb N}\mu(B_n)<\infty $. This allows an application of the Borel-Cantelli lemma and ensures that $\mu$-almost every $x\in X$ lies in at most finitely many $B_n$. In particular, there are major modifications only finitely many times in the definition of the value of function at $\mu$-almost every point in $X$. This implies the required convergence.

Gao, Jackson, Krohne and Seward \cite{gao_forcing_2015} have shown that under
very mild conditions this is not possible in the Borel setting. We will give a
simple proof for this result below.

Let $T$ be an action of a countable group $\Gamma$ on a compact space $X$ by
homeomorphisms.  We say that $(X,T)$ is \emph{minimal} if each orbit of the
action is dense.  Equivalently, for all open sets $U$ there is a finite set
$F\subset \Gamma$ such that $T^F(U)=X$.  A \emph{complete section} is a
measurable set which meets all the orbits of the action.

\begin{thm}\cite[Theorem 3.2]{gao_forcing_2015}\label{theorem: no tiling
theorem} Let $(X, T) $ be a free minimal action of a countable group $\Gamma$.
Let $B_n \subset  X$ be a sequence of Borel sets such that for all finite $F
\subseteq \Gamma$ and for all sufficiently large $n$, the set $\{x \in X \mid
T^\gamma(x) \in B_n$ for all $\gamma \in F\}$ is a complete section of a
comeager set.  Then the set $\{x\in X~:~x\mbox{ belongs to }B_n\mbox{ for
infinitely many }n\}$ is comeager.  \end{thm}

\begin{remark} \label{remark: tilings untiling}We make a few remarks:
\begin{enumerate}
\item Gao, Jackson, Krohne and Seward remark that such points can be found in
the free part of the full shift $(\mbox{Free}(\{0,1\}^{\Gamma}), \sigma)$.  This
uses an addional theorem by Gao, Jackson and Seward
\cite{gao_jackson_seward_2009} that there are minimal subsystems contained in
$(\mbox{Free}(\{0,1\}^{\Gamma}, \sigma)$.  We refer the reader to
\cite{aubrun_realization_2018} for an elementary proof of this fact using the
Lov\'asz local lemma.

\item In \cite{MR1716239} Prikhod\!\'{}ko and in \cite{MR2573000} \c{S}ahin proved that if $R_i\subset \mathbb Z^d; i \in \mathbb N$ is an infinite sequence of rectangles which are coprime then for any ergodic measure preserving $\mathbb Z^d$ system $(X, \mu , T)$  there is a Borel partition of a set of measure one: $B_i; i\in \mathbb N$ such that the shape of the connected components of the set $B_i$ are $R_i$. In fact given a positive probability vector $p_i; i \in \mathbb N$ they could ensure that the sets $\mu(B_i)= p_i$ as well. By ergodicity we have that the sets $B_i$ are complete sections for a set of measure $1$. This is no longer true in the Borel setting due to Theorem \ref{theorem: no tiling	theorem} if the minimum side length of $R_i$ is unbounded. This shows that results in ergodic theory are not necessarily true in the Borel context. Also have a look at the discussion in Section \ref{subsection: open ergodic to borel}.
\end{enumerate}
\end{remark}

A key fact that we will use in the proof is that if $B \subset X$ is a complete
section of a comeager set, then it cannot be meager. If it were meager it would
contradict Baire's theorem since it is a section of the set $C=\bigcup_{\gamma
\in \Gamma} T^{\gamma}(B)$ and the set $X\setminus C$ is meager.  In
addition, we will use a couple of standard facts: Borel sets $A$ have the Baire
property, that is, there is an open set such that its symmetric difference with
$A$ is meager. Secondly if $A$ is a set with the Baire property and for all open sets $O$, $O\cap A$ is not meager then $A$ is comeager.

\begin{proof} Let $A=\{x\in X~:~x\mbox{ belongs to }B_n\mbox{ for infinitely
many }n\}$.  Given a finite set $F\subset \Gamma$ we write 
	$$B_{n, F}=\{x\in X~:~T^{\gamma^{-1}}(x)\in B_n\mbox{ for all }\gamma \in F\}.$$

By assumption, we have that the set $B_{n, F}$ is a complete section of a
comeager set for large enough $n$ and hence cannot be meager.

Let $U \subset X$ be an open set. Choose a finite set $F\subset\Gamma$
such that $\bigcup_{\gamma\in F}T^{\gamma}(U)=X$. This implies that

\[\bigcup_{\gamma\in F}T^{\gamma}(U)\cap B_{n,F}=B_{n, F}.\]

Hence there is $\gamma \in F$ such that $T^{\gamma}(U)\cap B_{n,F}$ is
not a meager set. But if $x\in T^{\gamma}(U)\cap B_{n, F}$ then $T^{\gamma^{-1}}(x)\in U\cap B_n$.

Thereby we have that
$$T^{\gamma^{-1}}\left(T^{\gamma}(U)\cap B_{n,F}\right)\subset U \cap B_n$$
is not meager in $U$ for all sufficiently large $n$.

It follows that for all $N\in \mathbb N$
$$U \cap (\bigcup_{n\geq N}B_n)$$
is not meager. Since this is true for an arbitrary open set $U$ and the sets $\bigcup_{n\geq N}B_n$ have the Baire property we have that the set
	$\bigcup_{n\geq N}B_n$ is comeager. Finally we have that the set
	$$A=\bigcap_{N\in \mathbb N}(\bigcup_{n\geq N}B_n)$$
	is comeager. This completes the proof. 
\end{proof}

As an application, we have another theorem of Gao, Jackson, Krohne and Seward
\cite{gao_forcing_2015} which shows that if we have an action of $\zd$ where the
action of each generator is minimal with a sequence $B_n$ of Borel complete
sections whose connected components are rectangles with side lengths going to
infinity, then there is a comeager set of points which lie on the boundary of
infinitely many $B_n$.  In the setting of ergodic theory, this issue does not
arise provided we restrict our attention to a conull set (many such statements
can be found in \cite[Section 7]{MR4311117}).

\begin{cor} \cite[Theorem 5.1]{gao_forcing_2015}\label{cor: even the boundary
intersects} Let $d\geq 2$ and $(X, T)$ be a minimal $\mathbb Z^d$ system such
that subsystem with respect to the $\mathbb Z\times \{0\}^{d-1}$ is also
minimal.  Given a sequence of Borel sets $B_n \subset X$ with the following
properties:
	\begin{enumerate}
		\item $B_n$ is a complete section.
		\item The connected components of $B_n$ are finite rectangles such that if $v_n$ is the minimum side length of a rectangle in $B_n$, then $\lim_{n \to \infty} v_n = \infty$.
	\end{enumerate}
Then the set 
$$\{x\in X~:~x\mbox{ belongs to }\partial B_n\mbox{ for infinitely many }n\}$$
is comeager.
\end{cor}
\begin{remark}
As with Theorem \ref{theorem: no tiling theorem}, this corollary was also proved for the free part of the full shift. This follows immediately from our result because there are many subshifts satisfying the hypothesis of this theorem which are contained in the free part of the full shift. Here is an easy way to construct such an example. Take irrational numbers $\alpha_1, \alpha_2, \ldots, \alpha_d \in \mathbb R/\mathbb Z$ which are rationally independent and consider the $\mathbb Z^d$ action $T$ on $\mathbb R/ \mathbb Z$ given by 

$$T^{\gamma}(x)= x+ \sum_{k=1}^d\gamma_k\alpha_k$$
	and equivariant map
	$$\phi: (\mathbb R/\mathbb Z, T) \to (\{0,1\}^{\mathbb Z^d}, \sigma)$$ 
	given by $$(\phi(x))_{\gamma}= 1_{[0,1/2)}(T^\gamma(x))\mbox{ for }\gamma \in \mathbb Z^d.$$
	The closure of $\phi(\mathbb R/\mathbb Z)$ gives an example of such a shift space. It is a simple exercise to check that the shift space is minimal for the $\mathbb Z\times \{0\}^{d-1}$ subaction.
\end{remark}

\begin{proof} To apply Theorem \ref{theorem: no tiling theorem} we need Borel sets $C_n\subset \partial B_n$ such that
\begin{enumerate}
\item The sets $C_n$ are complete section for a comeager subset of $X$ for the subaction $\Z\times\{0\}^{d-1}$.
\item The connected components of $C_n$ with respect to the subaction $\Z\times\{0\}^{d-1}$ are intervals of size at least $v_n$ where $v_n\to \infty$.
\end{enumerate}

Once we prove this we are done. 
Let 
$$C'_n=\{x\in \partial B_n~:~T^{v\epsilon^1} \in \partial B_n\text{ for }1\leq v\leq v_n\}.$$
The orbit of each connected component of $B_n$ intersects $C'_n$ so we know that the sets $C'_n$ are complete sections for the entire $\Z^d$ action. Let
$$C_n=\bigcup_{v=1}^{v_n} T^{v\epsilon^1}(C'_n).$$
Since the sets $C_n$ are Borel (hence have the Baire property), there exist open sets $U_n$ such that $U_n \Delta C_n$ is a meager set. In addition since they are complete sections for the entire action, the sets $C_n$ are not meager and hence $C_n \setminus (U_n \Delta C_n)$ is also not meager. Because of the minimality of the subaction we have that
$$T^{\Z\times\{0\}^{d-1}}U_n=X$$
and hence
$$X\setminus (T^{\mathbb Z\times\{0\}^{d-1}}(C_n))\subset (T^{\mathbb Z\times \{0\}^{d-1}}(U_n\Delta C_n))$$
is meager. 
Thus the sets $C_n$ are complete sections for the subaction for the comeager set
$$X\setminus \cup_{n\in \N}(T^{\mathbb Z\times \{0\}^{d-1}}(U_n\Delta C_n)).$$
	This completes the proof.
\end{proof}

\section{Hyperfiniteness} \label{section: toast}

In this section, we formulate a witness to hyperfiniteness which reduces the
constructions of the maps from our main theorems to finite combinatorics.

\begin{defn} A finite subset $F$ of $\mathbb{Z}^d$ is $\alpha$-\emph{almost cube} if there are cubes $S$ and $S'$ of side
lengths $s$ and $\alpha s$ respectively with the same center such that $S
\subseteq F \subseteq S'$.  In this case, we say that $F$ has side length $s$.
\end{defn}

We reformulate Theorem 5.5 of \cite{marks_borel_2017} (announced in Gao,
Jackson, Krohne and Seward \cite{gao_forcing_2015}).

\begin{thm} \label{theorem: precake} Let $(X,T)$ be a free Borel $\Z^d$ system with $d > 1$
and let $0<\delta < 1$.  If $1<r_1 < r_2 \dots$ is a sequence of
natural numbers satisfying $12\sum_{j < k}r_j < \delta r_k$, then there is a
sequence of Borel sets $B_1, B_2, \dots$ such that
\begin{enumerate}
\item  the connected components $C$ of $B_j$ are coconnected $(1+\delta)$-almost cubes of side length $r_j$
\item for all $x \in X$, there is $k \in \N$ such that $x \in B_k$ and
\item if $C,D$ are connected components of $B_l$ and $B_m$ respectively with $l \leq m$, then $d(\partial C, \partial D) > r_l$.
\end{enumerate}
\end{thm}

For all $x\in X$ and for all $\gamma$
in $\Z^d$, there is $k$ such that $x, T^\gamma(x)$ are in the same connected
component of $B_k$. Indeed if $A$ is a connected component of $B_k$ for some $k$ then by conditions (2) and (3), $A\cup \partial A\subset B_l$ for some $l>k$. By repeating this argument $m$ times, we get that there exists $l'$ such that 
$$T^{[-m,m]^d}(A)\subset B_{l'}.$$
 
The proof is a routine modification of the proof in \cite{marks_borel_2017}.  We
sketch it noting that our sets $B_i$ are essentially $\bigcup D_i$ from
the proof there.

\begin{proof} Let $1 < r_1 < \dots$ be such that $12\sum_{j < k}r_j < \delta
r_k$.  Let $C_i$ be a sequence of Borel maximal $6r_i$-discrete sets (See
\cite{MR1667145} Theorem 4.2 or \cite{marks_borel_2017} Lemma
A.2 for details).  Suppose that we have defined $B_j$ for all $j < i$ such that additionally the connected components of $B_j$ are $4r_j$ apart.

We define $A^i_0 = \mathcal{B}_{r_i/2}(C_i)$ where $\mathcal{B}_r(C)$ is defined
to be $\{ T^\gamma(x) \mid x \in C$ and $\vert \gamma \vert_\infty \leq r\}$.
We define $A^i_j$ by induction on $j \leq i$ by the formula $A^i_j =
\mathcal{B}_{2r_{i-j}}(A^i_{j-1},B_{i-j})$ where $\mathcal{B}_r(A,B) = A \cup
\mathcal B_r(Z)$ where 
$$Z = \{x \in B \mid \text{ the connected component $C$ of $x$ in $B$
satisfies $d(A,C) \leq r$}\}.$$  Let $B_i = A^i_{i-1}$.  It is straightforward to show
by induction that $A^i_j \subseteq \mathcal{B}_{r_i/2 + 6\sum_{k<j}r_{i-k}}(C_i)$.
It is immediate that the connected components of $B_i$ are $(1+\delta)$-almost
cubes of side length $r_i$ with pairwise distance at least $4r_i$.

Next suppose that $0 <j \leq i$ and consider a connected component $C$ of $B_i$
and $D$ of $B_{i-j}$.  Then by our construction either $\mathcal{B}_{r_{i-j}}(D)
\subseteq C$ or $d(A^i_{j},D) > 2r_{i-j}$ (See \cite{marks_borel_2017} Lemma
A.5).  In the second case, it follows that $d(C,D) > 2r_{i-j} - \sum_{k < i-j}
6r_k > r_{i-j}$.

As in \cite{marks_borel_2017} we make the connected components $C$ of the $B_i$
coconnected by adding the finite connected components of $X \setminus B_i$ to
$B_i$.  It is straightforward to see that this retains the essential properties
of the $B_i$.  \end{proof}

We make use of the following consequence of Theorem \ref{theorem: precake}.

\begin{cor} \label{corollary: cake} Let $(X,T)$ be a free Borel $\Z^d$ system with $d > 1$,
$0<\delta<1$ and $n,M \in \mathbb{N}$.  If $r_1 < r_2 \dots$ is a sequence of
natural numbers satisfying $24\sum_{j<k}r_k < \delta r_k$ and $r_1 > 2n+12M/\delta$, then
there is sequence of Borel sets $B_1, B_2,
\dots$ such that

\begin{enumerate}
\item  the connected components $C$ of $B_j$ are coconnected $M$-grid unions which are $(1+\delta)$-almost cube of side length $r_j$.
\item for all $x \in X$, there is $i \in \N$ such that $x \in B_i$ and
\item if $C,D$ are connected components of $B_i$ and $B_j$ respectively, then $\partial^n C \cap \partial^n D = \emptyset$.
\end{enumerate}
\end{cor}

\begin{proof}  Let $n,M \in \mathbb{N}$ and $\delta>0$.  Let
$\delta'>0$ be such that $\delta/2<\delta'<\delta$ be such that $\delta'+2M/r_1 < \delta$.  Let $B_i'$
witness Theorem \ref{theorem: precake} for $\delta'$.

We can choose a Borel set $A_i \subset B_i'$ which meets every connected component of $B_i'$ exactly once. For all finite sets $F\subset \Z^d$ let $\alpha_F$ denote the minimal element of $F$ according to the lexicographic order. The lexicographic order induces a total order on orbits of elements of $X$.

The elements of $B_i'$ whose component has shape $F$ and are minimal in their connected component is given by 
$$(\cap_{\alpha \in F} T^{-\alpha+\alpha_F}(B_{i}'))\cap (\cap_{\alpha \in \partial F}T^{-\alpha+\alpha_F}(X\setminus B_{i}') ).$$
The union over all such sets $F$ gives us a Borel subset of $A_i\subset B_{i}'$ which meets every connected component of $B_i'$ exactly once.

For $\gamma \in
M\zd$ let $A_i^\gamma$ be the set of $x \in A_i$ such that $T^{\gamma +
[0,M)^d}(x)$ intersects the connected component of $x$ in $B_i'$.  Then let $B_i
= \bigcup_{\gamma \in M\zd}T^{\gamma+[0,M)^d}(A_i^\gamma)$.  It is clear that
the connected components of $B_i$ consist of $M$-grid unions.

Next, we have that $B_i \subseteq \mathcal{B}_M(B_i')$ and the connected
components of $B_i'$ are contained in boxes of side length $(1+\delta')r_i$, so
the connected components of $B_i$ are contained in boxes with side length
$(1+\delta')r_i+2M = (1+\delta'+2M/r_i)r_i < (1+\delta)r_i$.  Finally, it easy
to see that if $C$ and $D$ connected components of $B_l$ and $B_m$ respectively
with $l \leq m$, then $d(\partial C,\partial D) > r_l - 4M \geq r_1- 4M > 2n$.

As in the proof of Theorem \ref{theorem: precake} we can enlarge each $B_i$ so
that the connected components are coconnected while retaining their essential
properties. \end{proof}

\section{Tiling complements}\label{section: Tiling complements}
Anticipating the definition of property F, we show that certain differences of
$k$-grid unions can be tiled by ``nice" rectangles.  To start we prove a lemma
about tiling sets of the form $B \setminus A$ where $A$ is an $Nk$-grid union
and $B$ is a slightly larger $k$-grid union with a different offset.  These
tilings will be used to show that most of our spaces have property F.  We will
need some special properties for directed bi-infinite Hamiltonian paths in
Section \ref{Section: Directed Hamiltonian}. These properties appear in
Corollary \ref{corollary: Hamiltonian tiling}.

 A \emph{$k$-box} is a box whose sides are of length $k$. An \emph{almost $k$-box}
is a box such that all its sides are of length between $k$ and $2k$ and at most one of
its sides has length different from $k$ .

For $N,k \in \mathbb{N}$, let $R_{Nk}$ be the interval
$[0,Nk)$ in $\mathbb{Z}$ and $S_{Nk}$ be the interval $[-3k,(N+3)k)$ in
$\mathbb{Z}$.

\begin{lemma}\label{lemma: boundary_tiling}  Let $N,k \geq 10$ and let $L
\subseteq Nk\mathbb{Z}^d$ and $\gamma \in \mathbb{Z}^d$ with $ 0 \leq \gamma_i <
k$ for all $i \leq d$.  If we set $A = \bigcup_{\beta \in L} \beta + (R_{Nk})^d$
and $B = \bigcup_{\beta \in L} \beta + \gamma + (S_{Nk})^d$, then $B \setminus
A$ can be tiled by almost $k$-boxes.  \end{lemma}

We record a definition for use later.  For $i \leq d$, we define an operation
$S_i((a,b), \prod_{j \leq d} [\beta_j,\beta'_j))$ to be the set obtained by
replacing $[\beta_i,\beta'_i)$ by $(a,b)$ in the given product.  Further, if $R$
is a union of products of the form $\prod_{j \leq d}[\beta_j,\beta'_j)$, then we
define $S_i((a,b),R)$ to be the union of the sets $S_i((a,b),\prod_{j \leq
	d}[\beta_j,\beta'_j))$ over the products which make up $R$.

\begin{proof}
	We define a sequence of sets $F_i$ for $i \leq d$ by
	\[ F_i = \bigcup_{\beta \in L} \big(\prod_{j \leq i} \beta_j + \gamma_j + S_{Nk} \big)
	\times \big( \prod_{j>i} \beta_j + R_{Nk} \big) \]
	
	It is straightforward to see that $F_0 = A$ and $F_d = B$. For $i>0$, the
	definitions of $F_{i-1}$ and $F_i$ only differ on coordinate $i$ and on this
	coordinate we see that $\beta_i + R_{Nk} \subseteq \beta_i + \gamma_i +
	S_{Nk}$ since $ 0 \leq \gamma_i < k$.  It follows that $F_{i-1} \subseteq F_i$ and hence $A
	\subseteq B$.  Now $B \setminus A$ can be written as $\bigcup_{0<i \leq d} F_i
	\setminus F_{i-1}$ where we note that the sets in the union are disjoint.
	
	So to prove the lemma it suffices to prove the following claim:
	
	\begin{claim} For $i>0$, $F_i \setminus F_{i-1}$ can be tiled by almost $k$-boxes.
	\end{claim}
	
	We note that it is enough to tile the connected components of $F_i \setminus
	F_{i-1}$ by appropriate boxes.  We start by proving that for every $i$, $F_i$
	can be tiled with $k$-boxes.  Let $\gamma \upharpoonright i$
	be the vector given by the first $i$ elements of $\gamma$ followed by $d-i$
	zeros.  In particular $\gamma \upharpoonright 0$ is the vector of all zeros.  It
	is clear that $F_i$ can be written as the union of sets of the form $\gamma
	\upharpoonright i + \prod_{j\leq d} [\alpha_jk,(\alpha_j+1)k)$ where $\alpha \in
	\mathbb{Z}^d$.  This forms a tiling of $F_i$ by $k$-boxes.
	
	For a connected component $R$ of $\partial_i F_{i-1}$, the tiling of $F_{i-1}$
	by $k$-boxes induces a tiling $\mathcal{T}$ of $R$ by
	boxes with side length $1$ in the $i$ direction and $k$ in all the other directions.  Using the fact that $N>6$, the
	connected components of $F_i \setminus F_{i-1}$ are of the form
	$S_i((a_R,b_R),R)$ where $R$ is a connected component of $\partial_i F_{i-1}$
	and $(a_R,b_R)$ is an interval of length at least $2k$.  It is immediate from
	the definition of $S_i$ and the tiling $\mathcal{T}$ that $S_i((a_R,b_R), R)$ is
	tiled by boxes of the form $S_i((a_R,b_R), T)$ for $T \in \mathcal{T}$ which can be tiled by almost $k$-boxes.
	This completes the proof.  \end{proof}

We use the previous lemma to tile differences between a connected component $C$
of $B_i$ and connected components $D$ of $B_j$ for $j<i$ with $D \subseteq C$.

\begin{lemma} \label{lemma: tiling of Borel boxes}Let $N,k \geq 10$.  Suppose
that $B_1, B_2, \dots$ witnesses Corollary \ref{corollary: cake} with $M=Nk$ and
$n = 10k$.  Then for every $i$ and every connected component $C$ of $B_i$, there
is a tiling of $C \setminus \bigcup_{D \in Z} D$ with almost $k$-boxes where $Z$
is the set $\{D \mid D$ is a connected component of $B_j$ for $j<i$ and $D
\subseteq C\}$.  \end{lemma}

This lemma will be an immediate consequence of the following non-Borel version.

\begin{prop}\label{proposition: tiling of boxes}
Let $N,k \geq 10$.  Let $Z$ be a finite collection of $Nk$-grid unions such that the collection $D+[-10k, 10k]^d; D\in Z$ are disjoint and contained in a larger $Nk$-grid union $C$. Then $C\setminus \cup_{D\in Z}D$ can be tiled by almost $k$-boxes.
\end{prop}

\begin{proof}  We enlarge each $D \in Z$ to $D'$ such that $C \setminus
	\bigcup_{D \in Z}D'$ can be tiled with $k$-boxes.  The regions $D' \setminus D$ will be handled by Lemma \ref{lemma:
		boundary_tiling}.
	
	By shifting everything we can assume that for some $L^*\subset k \Z^d$ we have that
	$$C = \bigcup_{\beta \in L^*} \beta + [0,k)^d.$$

	Each $D$ is of the form $\bigcup_{\beta \in L} \beta + \gamma + [0,Nk)^d$ for
	some $\gamma \in \mathbb{Z}^d$ and $L \subseteq Nk\mathbb{Z}^d$.  Let $\gamma_D$
	be a vector of minimal length so that $\gamma + \gamma_D + L \subseteq
	k\mathbb{Z}^d$.  Note that $\Vert\gamma_D\Vert < k$.
	
	Let 
	$$D' = \bigcup_{\beta \in L} \beta + \gamma + \gamma_D +
	[-3k,(N+3)k)^d.$$
	By Lemma \ref{lemma: boundary_tiling}, $D' \setminus D$ can be
	tiled by almost $k$-boxes.  We note also that $C \setminus D'$ can also be tiled by
	$k$-boxes since $C$ has such a tiling and $D'$ can
	be written as a union of (a subset of) these tiles.
	
	Clearly, any element of $D'$ is distance at most $4k$ from an element of $D$.
	Hence if $D,E \in Z$, then $D'$ and $E'$ are disjoint.  Finally, $C \setminus
	\bigcup_{D \in Z} D'$ can be tiled by $k$-boxes, since the
	tiling of $C$ which was used to tile $C \setminus D'$ is independent of $D$.
\end{proof}

To apply these results in Section \ref{Section: Directed Hamiltonian} we will need some special properties of the tilings. To state these properties we need some more definitions. The \emph{faces} of a box $B$ are the sets $\partial_i^+(B)$ and
$\partial_i^-(B)$ for $1\leq i \leq d$. We say that two disjoint boxes $A$ and
$B$ \emph{meet} if are vectors $\alpha\in A$ and $\beta\in B$ such that their
difference is a standard generator of $\Z^d$. Given two disjoint boxes we say
that they meet \emph{face to face} if one has a face which is a translate of a face of the other by a standard generator. A tiling $c$ of a set $B\subset \Z^d$ is called \emph{face to face} if all boxes in $c$ which meet, meet face to face.

Given two disjoint boxes $A$ and $B$ we say that they meet \emph{almost face to
face} if they have faces $A'$ and $B'$ such that there is a standard generator
$\epsilon$ for which $A'\cap (B'+\epsilon)$ has cardinality at least $1/(12)^{d}$
fraction of the minimum of the  volume of $A'$ and $B'$. Given a tiling, we
regard the boxes in the tiling as vertices of a graph which are connected by an
edge if they meet almost face to face. We say that a tiling is \emph{almost face
to face} if the associated graph is connected. We call this graph the \emph{almost face to face graph}.

\begin{cor} \label{corollary: Hamiltonian tiling} Let $C$ and $D; D\in Z$ be as
in Proposition \ref{proposition: tiling of boxes} where $k\geq 100^d$. There is a tiling of
$C\setminus \bigcup_{D\in Z}D$ by almost $k$-boxes. The tiling of $C\setminus
\bigcup_{D\in Z}D$ has the following three properties:
	\begin{enumerate}
		\item For each $D\in Z$, the tiling of $C\setminus \bigcup_{D\in Z}D$ meets the tiling of $D$ by $k$-boxes, face to face on $\partial_1(D)$.
		\item The tiling of $C\setminus \bigcup_{D\in Z}D$ agrees with the tiling of $C$ by $k$-boxes on $\partial C$.
		\item The tiling of $C\setminus \bigcup_{D\in Z}D$ is almost face to face.
	\end{enumerate}	
\end{cor}

\begin{proof}
	For this proof we will trace the proof of Proposition \ref{proposition: tiling of boxes} making necessary observations as we go along.
	
	First the sets $D$ were expanded slightly to get sets $D'$ where the sets $C \setminus
	\bigcup_{D \in Z}D'$ can be tiled with $k$-boxes (which proves property $(2)$) while the regions $D' \setminus D$ (which contain the boundary of $D$) were handled by Lemma \ref{lemma: boundary_tiling}. In the proof of Lemma \ref{lemma: boundary_tiling} we wrote the set corresponding to $D'\setminus D$ as a union of sets $F_i\setminus F_{i-1}$ each of which were tiled separately. In the construction the set $F_1\setminus F_0$ contained $\partial_1(D)$ and we have that the tiling of $F_1\setminus D$ is obtained by taking the
	tiling of $D$ by $k$-boxes and extending it perpendicular to $\partial_1 D$. It follows that the tiling of $D'\setminus D$  that we constructed meets the natural tiling of $D$ by $k$-boxes face to face on $\partial_1(D)$. This proves Property $(1)$. 
	
	Now we will prove Property $(3)$. Let $B$ be a translate of $[1, k]^d$
centered in $[0,1]^d$. Consider the set 
	$$C':=\{\alpha~:~\alpha+B\subset C\setminus \bigcup_{D\in Z}D\}. $$ 
   $C'$ is a connected set (since it is just the set $C\setminus \bigcup_{D\in Z}D$ with a small portion of its boundary removed). Given any two boxes $B_1, B_2$ from the tiling, we can find a path $P$ in $C'$ with terminal vertices $\alpha, \beta \in C'$ contained in $B_1$ and $B_2$ respectively. Now suppose that $\gamma, \gamma+\epsilon^i\in P$ such that $\gamma$ is in the tile $B'$ and $\gamma+\epsilon^i$ is in another tile $B''$. Now all the tiles $\tilde B$ which meet $B'$ on $\partial_i^+(B')$ are connected in the almost face to face graph. In addition there has to be at least one tile among these tiles $\tilde B$ which meets $B'$ almost face to face. Consequently we have that $B'$ and $B''$ are in the same component of the almost face to face graph. Thus by following the path $P$ we must have that $B_1$ and $B_2$ are also in the same component and that the tiling is almost face to face.
	\end{proof}

We end the section with a lemma which is used to ensure points that we construct
are aperiodic.

\begin{lemma}\label{lemma: free part} Let $N, k> 100$.  Suppose that
$B_1, B_2, \dots$ witnesses Corollary \ref{corollary: cake} with $M=Nk$ and $n =
10k$ and sequence $r_i$ for $i \in \mathbb{N}$.  Then for every $i$ and $\beta \in \zd$ with $r_i/4 \leq
\vert \beta \vert \leq r_i/2$ there is a Borel set $E_i \subset B_i$ such that each connected component of $E_i$ is a cube $E$ of side length $k$ contained in a connected component $C$ of $B_i$ such that
\begin{enumerate}
\item $\partial^{k}E,\partial^{k}\beta+E \subseteq C \setminus \bigcup_{D \in Z}D$ and 
\item $\partial^{k}E, \partial^k \beta+E$ and $\partial^k D$ for $D \in Z$ are
disjoint
\end{enumerate}
where $Z= \{D \mid D$ is a connected component of $B_j$ for $j<i$ and $D
\subseteq C\}$.  \end{lemma}

\begin{proof}  Let $C$ and $Z$ be as in the lemma. Since the sets $\partial^{10k}(C)$ and $\partial^{10k}(D)$ for $D\in Z$ are disjoint it follows that the set
	$$\bar{C} = C \setminus(\partial^{3k}(\zd \setminus C) \cup
	\bigcup_{D \in Z} (D\cup \partial^{3k} D))$$
is connected.

Since
$\bar{C}$ is bounded, there is $x \in \bar{C}$ such that $T^\beta(x) \notin
\bar{C}$.  Let us now see why $\bar{C}\cap T^{-\beta}(\bar C)$ is non-empty. $C$ contains a cube of length $r_i-4k$ where $r_i-4k> |\beta|$ which is at distance $4k$ away from the boundary of $C$. Since it intersects its translate by $-\beta$ the same must occur for its inner $4k$ boundary as  well. Thus we have that
$$\partial^{4k}(\Z^d\setminus C)\setminus \partial^{3k}(\Z^d\setminus C)$$
must also intersect its translate by $-\beta$. But this set is disjoint from $D\cup \partial^{3k} D$ for all $D\in Z$. Thus we have that $\bar C\cap T^{-\beta}(\bar C)$ is non-empty. Let $y\in \bar C\cap T^{-\beta}(\bar C)$.  Since $\bar{C}$ is connected there is a path $p$ from $x$ to $y$
contained in $\bar{C}$.  Let $w$ be the first vertex on the path $p$ for which
$T^\beta(w) \in \bar{C}$.  Let $E$ be a cube of side length $k$ centered at
$w$.  It is clear that $E$ satisfies the conditions of the lemma. Since there are only finitely many choices of $E$ given the shapes of various connected components involved it follows that there exists a Borel set $E_i$ with the required properties. \end{proof}

\section{Property F} \label{section: propertyF}
In this section, we isolate a mixing condition called property $F$ and prove
that any free Borel $\Z^d$ system admits a Borel factor map into the free part
of a symbolic system with property F.  A \emph{family of patterns} is a
collection $(P_C\subset \mathcal{L}(X, C); \text{$C$ is a $k$-grid union})$ such
that $P_C=\sigma^\alpha(P_{C-\alpha})$ for all $\alpha\in \zd$. 

\begin{defn}\label{definition: Property F}
A family of patterns $(P_C; \text{$C$ is a $k$-grid union})$ has \emph{property $F$ with gap $m$ }if the following two properties are satisfied. 
\begin{enumerate}
	\item For every
	\begin{enumerate}
		\item $k$-grid unions $C, C_1, C_2, \ldots, C_r$ and
		\item patterns $c_i\in P_{C_i}$,
	\end{enumerate}
	such that
	$C_1+[-m,m]^d, C_2+[-m,m]^d, \ldots, C_r+[-m,m]^d$ are disjoint and contained in
	$C$, there exists $c\in P_C$ such that $c|_{C_j}=c_j \text{ for }1\leq j
	\leq r$.
	\item If $x$ is a configuration such that there is an increasing
	sequence of $k$-grid unions $C_k$ which cover $\Z^d$ and $x|_{C_k}\in
	P_{C_k}\text{ for all }k$ then $x\in X$.
\end{enumerate}
\end{defn}

We note that each $k$-grid union $C,C_1,C_2, \dots$ may use a
different offset $\gamma$.  We also note that property $(2)$ is satisfied by every subshift.  We will need this to prove that
the space of bi-infinite Hamilton paths has Property F.

\begin{thm} \label{theorem: property F factor} If $(Y,T)$ is a free Borel $\Z^d$ system
and $(X,S)$ is a symbolic system satisfying Property F, then there is an equivariant
Borel map from $Y$ to the free part of $X$.
\end{thm}

\begin{proof}  Suppose that $(X,S)$ is a symbolic system contained in $\{1,2,\ldots, l\}^{\zd}$
which has property $F$ relative to $k$-grid unions with gap $m$.  Let $P_C$ for
$k$-grid unions $C$ be the witnessing set of patterns.  Let $B_1,B_2
\dots$ witness Corollary \ref{corollary: cake} for $(Y,T)$ with $M=n=10k$,
$\delta=1/2$ and any sequence $r_1 < r_2 < \dots$ satisfying the hypothesis of
the corollary.

We define a function $\Phi$ using property $F$ as follows.  Let $\beta_i$ for $i
\in \mathbb{N}$ be a sequence of vectors such that
\begin{enumerate}
\item for all $i \in \mathbb{N}$, $r_i/4 \leq \vert \beta_n \vert \leq r_i/2$
and
\item for all $\gamma \in \zd$, there is $j \in \mathbb{N}$ such that $j\gamma =
\beta_i$ for some $i \in \mathbb{N}$.
\end{enumerate}
Such a sequence is easily constructed by induction.  We define a function $\Phi$
which captures our applications of property F.  The inputs are
$(1+\delta)$-almost cube $M$-grid unions $C$ of side length $r_i$ in $\zd$ and
patterns $d_1, \dots d_r$ over $(1+\delta)$-almost cubes $D_1, \dots D_r$ of
side lengths $r_j$ for various $j<i$ such that $d_i \in P_{D_i}$ and
$\partial^{10k}D_1, \dots \partial^{10k}D_r$ are disjoint subsets of $C$.  Given
such an input, we can
apply Lemma \ref{lemma: free part} with
$\beta = \beta_i$, to find a cube $E$ of side length $k$ as in the lemma.  We
then choose $e \in P_E$ and $e' \in P_{\beta_i+E}$ such that
$e'\neq\sigma^{\beta_i}(e)$. Finally we let $d=\cup_{i=1}^rd_i\cup e\cup e'$ and
$\Phi(C,d)$ be a
pattern obtained by applying property F to the patterns $d_1, \dots d_r,e,e'$
inside $C$.  We can assume that $\Phi$ commutes with the shift operation.

We will now define a sequence of Borel functions $f_n:\cup_{i=1}^n B_i \to \{1, \dots l\}$ by
induction.

To initialize the construction we let $B_0 = \emptyset$ and $f_0$ be the trivial
function.  Suppose that we have constructed $f_n$ for some $n$.  Let $y \in
B_{n+1} \setminus \cup_{i=1}^nB_i$.  We define the following:
\begin{enumerate}
\item $C_y= \{ \gamma \mid y,T^\gamma(y)$ are in the
same connected component of $B_{n+1}\}$, 
\item $D_y = \{ \gamma \in C_y \mid T^\gamma(y) \in B_i$ for some $i \leq n \}$
\item $d_y$ is the function with domain $D_y$ such that $d_y(\gamma)
= f_n(T^\gamma(y))$.
\end{enumerate}

Let $f_{n+1}(y) = c_y(\bar{0})$ where $c_y = \Phi(C_y,d_y)$. For $y\in \cup_{i=1}^nB_i$ we define $f_{n+1}(y)=f_n(y)$. 
Since $C_y, D_y$ and $d_y$ are Borel functions on $B_{n+1}\setminus \cup_{i=1}^nB_i$ with finite range and $\Phi$ is a map between finite sets it follows
that $f_{n+1}$ is Borel.  Moreover, since the assignments of $C_y,D_y$ and $c_y$
are equivariant and the operation $\Phi$ commutes with the shift, we have that
for a connected component $C$ of $B_{n+1}$, the pattern given by fixing $y \in
C$ and taking $\gamma \mapsto (\cup_{i=1}^{n+1}f_{i})(T^\gamma(y))$ for $\gamma$ such that
$T^{\gamma}(y) \in C$ is in $P_C$. Further by our choices of $e$ and $e'$ above we have that for all $y \in B_{n+1}$ there exists $\gamma $ such that 
$$f_{n+1}(T^\gamma(y))\neq f_{n+1}(T^{\gamma+\beta_{n+1}}(y)).$$
Thus the map $\hat{f}: Y \to X$ given by 
$$\hat{f}(y)(\gamma) = f_{n}(T^{\gamma}(y)) \text{ for $n\in \N$ and all $y\in B_n$}$$ 
is equivariant and maps to the free part of $X$. 
\end{proof}

In Section \ref{section: embedding} we will prove that if $(Y,
T)$ is a symbolic space with appropriate entropy then we can actually find an
embedding.

\section{Markers} \label{section: markers}
The entropy of the collection $\mathcal C =(C_B; B\text{ is a
$k$-grid union})$ is defined by \[ h(\mathcal C)=\liminf \frac{1}{(nk)^d}
\log{|C_{[1,nk]^d}|}.\] The limit exists under the assumption of property $F$ (but we will not need this fact) and is
non-zero (except in the trivial case of each $C_B$ being a singleton). The proof
follows from the proof of \cite[Lemma 5.3]{MR4311117}.

We say that patterns $a, b\in \mathcal{L}(X)$ clash if there is a site
$\gamma$ such that $a_\gamma\neq b_\gamma$.  A collection $\mathcal M= (M_n\subset \mathcal L(X, [1,nk]^d); n\in \N)$ is said to have \emph{the marker property} with tolerance $r$ if for all $a, b \in M_n$ and $|\gamma|< nk-r$, we have that $\sigma^{\gamma}(a)$ clashes with $b$ unless $\gamma=\m 0$ and $a=b$.

We say that the collection $\mathcal M$ has \emph{the marker property $F$} if it
has the marker property witnessed by $M_n$ for $n \in \mathbb{N}$ and there
exists a collection $\mathcal C =(C_B; B\text{ is a $k$-grid union})$ with
property $F$ such that $M_n\subset C_{[1,nk]^d}\text{ for all }n\in \N$.

\begin{thm}\label{theorem: markers come for free}
	If $X$ is a shift space which has a collection $\mathcal C$ with
property $F$ then it is has a collection $\mathcal M$ with marker property $F$
and $h(\mathcal C)=h(\mathcal M)$.
\end{thm}

We will follow the strategy as in \cite{MR4311117}. The situation is much simpler here because we are dealing with shift spaces, property $F$ is much stronger than flexibility (being used in that context) and the dimension $d\geq 2$. Marker constructions have a long history in ergodic theory and similar ideas can be found throughout the literature improvised according to context.

We will need the following lemma.

\begin{lemma}[No small periods]\label{lemma: no period}
Let $X$ be a shift space, $\mathcal C=(C_n; n \in \N) $ where $C_n\subset
\mathcal L(X, [1,nk]^d)$ be a collection with entropy $h(\mathcal C)>0$. Let
$r\in \N$. There exists a collection $\mathcal C'=(C_n'; n \in \N)$ where
$C_n'\subset C_n$ such that $h(\mathcal C')= h(\mathcal C)$ and for all $a\in
C_n'$ and $\gamma\in [-r,r]^d\setminus \{\bar 0\}$, $a$ clashes with
$\sigma^{\gamma}(a)$.  \end{lemma}

This lemma is very similar to \cite[Proposition 12]{hochman_2010} with
essentially the same proof. However, we cannot use the statement directly, so we
repeat the argument.

\begin{proof}
Let $\gamma\in [-r,r]^d\setminus \{\bar 0\}$.  Consider the set of patterns
$$P_{n,\gamma}=\{a\in \mathcal C_n~:~\text{$a$ does not clash with
$\sigma^\gamma(a)$}\}.$$  
The elements of $P_{n, \gamma}$ are periodic with period $\gamma$ and thus if $a, b\in P_{n,\gamma}$ agree on the set
$[1,nk]^d \setminus [3r, nk-3r]^d$, then $a=b$.

Hence there is an injection from $P_{n,\gamma}$ to $\mathcal L (X,
[1,nk]^d \setminus [3r, nk-3r]^d)$ and it follows that $|P_{n,\gamma}|<c\exp^{c'
n^{d-1}}$ for some constants $c, c'>0$ which only depend on $r$.

Thus the collection $\mathcal C'=(C_n'; n \in \N)$ given by $$C_n'=\{a\in
C_n~:~a\notin P_{n, \gamma} \text{ for }\gamma\in [-r,r]^d\setminus \{\bar
0\}\}$$ satisfies the conclusion of the lemma.

\end{proof}

We are now prepared for the construction of the markers.

\begin{proof}[Proof of Theorem \ref{theorem: markers come for free}]
Let $\mathcal C$ be a collection with property $F$ and gap $m$.  By Lemma \ref{lemma: no period} there exists $n\in \N$ and $a'\in C_{[1,nk]^d}$ such that $a'$ clashes with $\sigma^{\gamma}(a')$ for 
for all $\gamma\in [-1,1]^{d}\setminus\{\m 0\}$. The collection $\mathcal M=(M_t; t\in \N)$ with the marker property will consist of patterns whose boundaries are densely packed by elements of $a'$ but with different phases. This is to ensure that given two patterns from $M_t$ for some $t$ if their boundaries intersect substantially then there must be a clash because of the various phases.

First we divide the support of the elements of $M_t\subset C_{[1,tk]^d}$ into thin annuli and a large central region.
\begin{enumerate}[(1)]
\item For $t$ large enough we consider concentric annuli in $[1,tk]^d$
with width $100(m+nk)$ labelled $A_\gamma$ for $\gamma\in [0,m+nk]^{d}$ starting
from the outside towards the inside and leaving gaps of width $3(m+nk)$ in-between.
\item The complement of (the union of) the annuli $A_\gamma$ is a box of the form $[1+\theta k, (t-\theta)k]^d$ for some constant $\theta \in \N$ which independent of $t$.
\end{enumerate}

Having defined these regions, we also define some patterns that will be used in
the definition of $M_t$.  Using Lemma \ref{lemma: no period}, we find $a'' \in
C_{[1,sk]^d}$ for some $s$ such that for all $\gamma$ in $[-2k\theta,2k\theta]$,
$\sigma^\gamma(a'')$ clashes with $a''$.  Next we consider a configuration $x$ given by
$\sigma^{\gamma}(x)|_{[1, nk]^d}=a'$ for all $\gamma\in (2m+n)k\Z^d$.

Now define $M_t$ as the set of all patterns $b$ in $C_{[1,tk]^d}$ with the following properties:
\begin{enumerate}[(i)]
	\item $b|_{A_\gamma}=\sigma^{\gamma}(x)|_{A_\gamma}$ for all $\gamma\in
[0,m+nk]^d$ with the additional condition that copies of $a'$ in
$\sigma^\gamma(x)$ which lie on the boundary of $A_\gamma$ are included. \label{item: anulli tiling}
	\item $b|_{[1+\theta k, (t-\theta)k]^d}$ is a translate of a pattern
$c\in C_{[1, (t-2\theta)k]^d}$ such that $c$ extends $\sigma ^{\bar{m}} ( a'')$ where
$\bar{m}$ is the constant $m$ vector.\label{item: small translates}
\end{enumerate}

By property F,  a straightforward entropy calculation shows that $h(\mathcal{M}) =
h(\mathcal{C})$. To finish we will show that for all distinct $e,f\in M_t$, $f$ clashes with
$\sigma^{\gamma}(e)$ for all $\gamma \in [-(t-2\theta)k, (t-2\theta)k]^d$.
There are two cases.  Firstly if $\vert \gamma \vert_\infty \leq 2\theta k$, then
Condition (\ref{item: small translates}) from the definition of $M_t$ ensures that
$f$ clashes with $\sigma^\gamma(e)$.  Specifically, each of $f$ and
$\sigma^\gamma(e)$ contain copies of $a''$ which clash by the choice of $a''$.

 Secondly, if $\vert \gamma \vert_\infty > 2\theta k$, then we work with Condition
(\ref{item: anulli tiling}) find a clash between $f$ and $\sigma^\gamma(e)$.  For
each $i \leq d$, let $\alpha_i$ be the remainder of $\gamma_i$ when divided by
$m+nk$.  Next find $\delta$ with $\vert \delta \vert_{\infty} < m+nk$ such that $\vert
\delta+\alpha - \gamma \vert_\infty = 1$.  We consider the restriction of $e$
to the annulus $A_\delta$ and note that by our assumption on $\vert \gamma
\vert_\infty$, $A_\delta$ intersects $\sigma^\gamma(A_\delta)$ in a rectangle of side
length $100(m+nk)$.  Moreover, on this rectangle $\sigma^\gamma(e)$ is equal to
$\sigma^\gamma(\sigma^\delta(x)) = \sigma^\alpha(\sigma^\delta(x)) =
\sigma^{\alpha+\delta}(x)$.  By the choice of $\delta$ and the size of the
intersection of $A_\delta$ and $\sigma^\gamma(A_\delta)$, there are copies of $a'$ in
$f$ and $\sigma^\gamma(e)$ which are shifts of each other by a vector of norm
$1$.  So $f$ and $\sigma^\gamma(e)$ clash our the choice of $a'$.  \end{proof}

\section{Constructing the embedding}\label{section: embedding}

We will now strengthen Theorem \ref{theorem: property F factor} in the case of
shift spaces. 

 \begin{thm} \label{theorem: property F embedding subshift} Let
$(X,S)$ be a symbolic space with a collection $\mathcal C$ with Property F and $(Y,
T)$ is a shift space such that $h(Y, T)<h(\mathcal C)$. Then there is an
equivariant Borel embedding from the free part of $Y$ to $X$.  \end{thm}

Suppose $X$ is a symbolic space with entropy $h(X)$ and $Y$ is a shift space with a
collection $\mathcal C=(C_n; n\in \N); C_n\subset \mathcal L(X, [1,nk]^d)$ with
marker property $F$, entropy $h$, gap and tolerance $r/3$ (for some $r\in 3\N$).  We can assume
that each $C_n$ satisfies Lemma \ref{lemma: no period} for shifts of norm at
most $r$.

Since $h(Y, T)<h(\mathcal C)$, by a standard calculation involving the topological entropy, there are an $n \in
\N$ and an injective function
$$ \Psi: \mathcal L(X, [1,nk+2r]^d)\cup \mathcal L(X, [1,nk+2r+1]^d) \to
\sigma^{\bar{r}}(C_n) $$
where $\bar{r}$ is the constant $r$ vector.  For ease of notation below, we can
extend $\Psi$ to all shifts of the domain in a way that commutes with the shift
operation.

We note that $\sigma^{\bar{r}}(C_n)$ is a set of patterns on $[r+1,nk+r]^d$ and
$\partial^{r} [r+1,nk+r]^d$ is contained in $[1,nk+2r]^d$ (and
hence in $[1,nk+2r+1]^d$ as well).

Let $B_1, B_2, \dots$ be a sequence of Borel sets witnessing Corollary
\ref{corollary: cake} with parameters appropriate to apply the proof of Theorem
\ref{Theorem: Main} for the space of tilings using tiles $[1,nk+2r]^d$ and
$[1,nk+2r+1]^d$.  Let $H:X \to X(\{[1,nk+2r]^d,[1,nk+2r+1]^d\})$ be the Borel factor
map obtained from the theorem. In particular $[1,nk+2r]^d$ and $[1,nk+2r+1]^d$ tile $B_n$ for all $n$. 

Let $R_x \subseteq \Z^d$ be the tile of $\bar{0}$ in $H(x)$.  Let $\gamma_x\in \Z^d$ be such that $R_x$ is either $\gamma_x +[1,nk+2r]^d$ or
$\gamma_x+[1,nk+2r+1]^d$.  Let $B_0$ be the set of $x$ such that $\bar{0} \in
\gamma_x + [r+1,nk+r]^d$.  

We define a function $f_0$ on $B_0$ by $f_0(x) = c(\bar{0})$ where $c =
\Psi(x\vert_{R_x})$.  Since the connected components of $B_0$ are at least $2r$ apart
and at least $2r/3$ distance from $\partial B_n$ for all $n$, by using property $F$ we can extend $f_0$ defined on $B_0$ to $f_n$ defined on $\cup_{m=0}^nB_m$ for all $n$ as in the proof of Theorem \ref{theorem: property F factor}. In fact it is simpler here because we do not need to use Lemma \ref{lemma: free part}. At the end of the
construction, we let $f = \bigcup_{j \in \mathbb{N}}f_j$ and $\hat{f}$ be given
by $\hat{f}(y)(\gamma) = f(T^\gamma(y))$.  We have as in Theorem \ref{theorem:
property F factor} that the resulting maps $f$ and $\hat{f}$ are Borel.

We would like to prove that $\hat{f}$ is an embedding of $Y$ into $X$.  Suppose
that $y,y' \in free(Y)$.  We associate to $y$ an infinite pattern $p = \gamma
\mapsto f_0(T^\gamma(y))$ for $\gamma$ such that $T^\gamma(y) \in B_0$ and
similarly associate $p'$ to $y'$.  If $\dom(p) = \dom(p')$, then $\hat{f}(y) =
\hat{f}(y')$ implies $p=p'$ which implies $y=y'$, since the map $\Phi$ is injective.

So we can assume that $\dom(p) \neq \dom(p')$.  We fix a connected component $R$
of $\dom(p)$ which is not equal to any connected component of $\dom(p')$.  Since
sets of the form $R' \cup \partial^{r+1}R'$ for connected components $R'$ of
$\dom(p')$ cover $\Z^d$, there is some connected component $R'$ at distance
at most $r+1$ from the center of $R$.  It follows that $R$ is a shift of $R'$ by
some $\gamma$ of norm at most $r+1 + nk/2 < nk -r$.  By the tolerance of the
set of markers we have that $p \vert_R$ clashes with $p' \vert_{R'}$ and hence
$\hat{f}(y) \neq \hat{f}(y')$.

\section{Hom-shifts have property $F$} \label{Section: Hom-shifts}
In this section we will prove that hom-shifts have property $F$.  The argument we present is essentially contained in \cite[Section 9]{MR4311117} but needs some restatement because of our setting.

\begin{thm}\label{theorem: hom shifts have property F} Let $\text{Hom}(\Z^d, H)$
be a hom-shift where $H$ is a connected graph which is not bipartite. Then it
has property $F$ for a sequence of patterns with entropy equal to the
topological entropy of the hom-shift.  \end{thm} For this proof we denote the
interior boundary of a set $C$ by $$\partial_{\text{int}}C=\{\alpha\in
C~:~\alpha\text{ is adjacent to a vertex in $\zd\setminus C$}\}.$$ Further by
parity of a coordinate $\alpha\in \Z^d$ we mean the parity of the sum of its
coordinates, which we denote as $\parity(\alpha)$.  \begin{proof} Choose
vertices $v,w\in H$ which form an edge. For every connected set $C$ fix a site
$\alpha_C\in C$ such that $$\alpha_{\beta+C}=\beta+\alpha_C.$$ 
Consider the set of patterns $$P_{C}=\left\{a\in \hom(C,
H)~:~\text{ for all }\beta \in \partial_{\text{int}} C, a_{\beta}=\begin{cases}
v&\text{ if $\parity(\alpha_C)=\parity(\beta)$}\\ w&\text{ otherwise }.
\end{cases}\right\}.$$ We proved in \cite[Proposition
9.2]{MR4311117} that the entropy of this collection of patterns is
equal to the entropy of hom-shift. We are left to prove property $F$ for the
collection. Since $H$ is connected and not bipartite, we can fix an even integer $N$
so that there is a path of length $N$ from $v$ to $w$ and some $N'>N$.  Let $v_0
\dots v_N$ be such a path with $v_0 = v$ and $v_N=w$.

Now let $C, C_1, C_2, \ldots, C_r$  be connected subsets such that
$C_1+[-N',N']^d, C_2+[-N',N']^d, \ldots, C_r+[-N',N']^d$ are disjoint and
contained in $C$. Further let $c_i\in P_{C_i}$ be given. We need to find $c\in
P_C$ such that $c|_{C_i}=c_i$.

Consider connected subsets $C_i'=\partial^N C_i$ for all $1\leq i \leq r$.  Given
$1\leq i \leq r$, we define patterns $c'_i\in \mathcal L(\hom(C_i', H))$
extending $c_i$ depending on the parity of $\alpha_C$ and $\alpha_{C_i}$.  If
$\alpha_C$ and $\alpha_{C_i}$ have the same parity, then for $\alpha \in C'_i
\setminus C_i$ we set $c'_i(\alpha) = v$ if $\parity(\alpha_C)=\parity(\alpha)$
and $c_i'(\alpha) = w$ otherwise.  If $\alpha_C$ and $\alpha_{C_i}$ have
different parity, then to define $c'_i(\alpha)$ we let $t$ be such that
$\alpha \in \partial^t C_i\setminus \partial^{t-1}C_i$ and let $c_i'(\alpha) =
v_{t-1}$ if $\parity(\alpha) = \parity(\alpha_{C_i})$ and $c_i'(\alpha) = v_t$
otherwise.

It follows that for the patterns $c'_i$ we must have that for all $\alpha\in
\partial_{\text{int}}(C'_i)$
$$c'_i(\alpha)=\begin{cases}
\text{a vertex adjacent to $w$}&\text{ if $\parity(\alpha_C)=\parity(\alpha)$}\\
w&\text{otherwise}.
\end{cases}$$
and the pattern $c\in \mathcal L(\hom(\zd, H))$ extending the $c_i'$ defined on
$\alpha \in C\setminus \cup_{1\leq i \leq r}C_i'$ by $c(\alpha) = v$ if
$\parity(\alpha) = \parity(\alpha_C)$ and $c(\alpha) = w$ otherwise is what we
set out to obtain.\end{proof}

\section{Rectangular tilings have property $F$}\label{Section: Tiling shift}
In this section we will show that coprime rectangular shifts have property $F$. The argument comes from \cite[Section 10]{MR4311117}.
\begin{thm}\label{theorem: rectangular shifts have property F}
	Let $X(\mathcal T)$ be a coprime rectangular shift. Then it has property $F$. If $X(\mathcal T)$ is the set of domino tilings of $\Z^2$ then it has property $F$ for a sequence of patterns with entropy equal to the topological entropy of $X(\mathcal T)$. 
	\end{thm}
\begin{proof} Let $k$ be the product of the side lengths of all the boxes in $\mathcal T$. For every $10k$-grid union $C$ we let
	$$P_C=\{\text{tilings of $C$ by elements of $\mathcal T$}\}.$$
	In the case of domino tilings of $\Z^2$ in the proof of \cite[Theorem 10.3]{MR4311117} we showed by following \cite{kasteleyn1961statistics,MR1815214} that the entropy of this collection of patterns is the topological entropy of the space of domino tilings. We are now left to prove property $F$ for this collection.
	
 Let $C, C_1, C_2, \ldots, C_r$  be $10k$-grid unions such that  $C_1+[-10k,10k]^d, C_2+[-10k,10k]^d, \ldots, C_r+[-10k,10k]^d$ are disjoint and contained in $C$ with offsets $\alpha, \alpha_1, \ldots, \alpha_d \in [0,10k)^d$. Let $a_i\in P_{C_i}$ for all $1\leq i \leq r$. We will now construct $a\in P_{C}$ such that $a|_{C_i}=a_i$ for all $1\leq i \leq r$. 
 
Now we write 
$$C_i= \bigcup_{j=1}^{r_i}\alpha_i+\gamma_{i,j}+[0, 10k)^d$$
for some set 
$$\{\gamma_{i,j}~:~1\leq j \leq r_i\}\subset 10k\zd.$$

Choose $\alpha_i'$ such that $\alpha-\alpha'_i\in k\zd$ and $\alpha'_i\in \alpha_i+[0,k)^d$. We now consider the sets
$$C'_i=\bigcup_{j=1}^{r_i}\alpha'_i+\gamma_{i,j}+[-3k, 13k)^d.$$
The sets $C'_i$ are disjoint (since the sets $C_i+[-10k,10k]^d$ are disjoint). By Lemma \ref{lemma: boundary_tiling}, $C'_i\setminus C_i$ can be tiled by boxes each of whose side length is at least $k$ and at most one side different from $k$. Further by our choice of $\alpha'_i$ we have that 
$$C\setminus \bigcup C'_i$$
can be tiled by boxes all of whose sides lengths are $k$. Since each side length of a box in $\mathcal T$ is a factor of $k$ it follows that $\mathcal T$ can tile $C\setminus \bigcup C'_i$. Fix such a tiling $c'$ on $C\setminus \bigcup C'_i$. We are left to show that if $B$ is a box such that each of whose side length at least $k$ and at most one side different from $k$ then it can be tiled by elements of $\mathcal T$. This is an easy consequence of the solution to the well-known Diophantine Frobenius problem or the coin problem. Assume that $B=[0,N)\times[0,k)^{d-1}$ where $N\geq k$. Let 
$$\mathcal T=\{T_1, T_2, \ldots, T_s\}$$
and $l_i$ be the length of $T_i$ in the the $\epsilon^1$ direction.  Since the highest common factor of $l_i; 1\leq i \leq s$ is $1$ it follows that $N$ can be written as a sum of multiples of $l_i$ and thus $B$ can be tiled by layers perpendicular to the direction $\epsilon^1$ where each layer (of width a suitable multiple of $l_i$) is tiled by a single element of $\mathcal T$. From this we obtain a tiling of $C'_i\setminus C_i$ by elements of $\mathcal T$ which we call $a'_i$. The tilings $a'_i; 1\leq i \leq r$ and $c'$ give us a tiling of
$$C\setminus \bigcup_{i=1}^r C_i.$$
Together with the tilings $a_i$ of $C_i$, we get a tiling $a$ of $C$ which we were looking for.
	\end{proof}

\section{Property $F$ For Directed Bi-infinite Hamiltonian paths}\label{Section: Directed Hamiltonian}
\subsection{The setup}
In this section we prove that directed bi-infinite Hamiltonian paths have
property $F$. For the purpose of this section, we think of a
Hamiltonianian cycle (path) $c$ as a function where the domain is the set of
vertices of the cycle (path) and a directed edge is drawn from $\alpha$ to
$\beta$ if $\alpha+c(\alpha)= \beta$.

Our basic method for connecting Hamiltonian cycles is given in the following
proposition.

\begin{prop} \label{proposition: Halmiltonian join}Let $C$ and $D$ be finite subsets of $\Z^d$ and $c_1$ and $c_2$ be directed Hamiltonian cycles on them.
	Suppose there is an edge $(\alpha, \beta)$ in $c_1$ and an edge $(\gamma, \delta)$ in $c_2$ such that
	\begin{enumerate}
		\item $\alpha$ is adjacent to $\delta$.
		\item $\beta$ is adjacent to $\gamma$.
	\end{enumerate}
	Then by replacing $(\alpha, \beta)$ by $(\alpha, \delta)$ and $(\gamma, \delta)$ by $(\gamma, \beta)$ we get a Hamiltonian cycle on $C\cup D$. 
\end{prop}

The proof is obvious and left to the reader.

\begin{thm}\label{theorem: Hamiltonain_F}
The space of bi-infinite directed Hamiltonian paths has property $F$.
\end{thm}

Let $k>1000\times 100^d$ be an even integer and $C$ be a $100k$-grid union. With
any such grid union, there is a unique tiling of $C$ with $k$-boxes. Choose a
$k$-box $C'$ from this tiling such that $\partial^+_1(C')\subset
\partial^+_1(C)$.  We let $P_C$ to be the set of directed Hamiltonian paths
which start at a vertex $\alpha$ and end at vertex $\alpha+\epsilon^2$ such that
$\alpha, \alpha+ 16 \epsilon^2\in \partial^+_1(C')\text{ but }\alpha+17
\epsilon^2\notin \partial^+_1(C').$

We will prove that this collection of patterns satisfy property $F$. For this let $C, C_1, C_2, \ldots, C_r$ be $100k$-grid unions and patterns $c_i\in P_{C_i}$ such that $C_1+[-10k, 10k]^d$, $C_2 + [-10k, 10k]^d$, \ldots,$C_r+[-10k, 10k]^d$ are disjoint and contained in $C$. We will extend the $c_i$ to a path $c\in P_C$.

To begin we will consider some very special directed Hamiltonian paths on two
dimensional boxes. These paths are constructed especially so that they can be easily
put together to form directed Hamiltonian paths on almost $k$-boxes and that
directed Hamiltonian paths on adjacent almost $k$-boxes can be connected.  These
connections will be made using Proposition \ref{proposition: Halmiltonian join}.

\begin{lemma}\label{lemma: Hamiltonian on layer}
	Let $l, m>90 d$ such that $lm$ is an even integer. There is a directed Hamiltonian cycle $c_{l,m}$ on $[1,l]\times [1,m]$ with the following features.
	\begin{enumerate}
		\item The Hamiltonian cycle moves in the clockwise direction on all the edges of the boundary except one.
		\item At distance greater than $4$ from the boundary the edges
of cycle are moving either in the $\epsilon^1$ direction if they are at even
distance from the top edge or the $-\epsilon^1$ direction if they are at odd distance from the top edge.
	\end{enumerate}
\end{lemma}
To aid the reader, various cases (according to the parity of $l$ and $m$) have been illustrated in Figure \ref{figure:Hamiltonian_1}.
\begin{proof}
	We begin by drawing a clockwise oriented cycle on the boundary of
$[1,l]\times [1,m]$. Now we delete one of these edges to move the interior of
box. The path is extended inside using a path that goes back and forth across
the rectangle in the
$\epsilon^1$ or $-\epsilon^1$ direction working from top to bottom. This works well except when $m$ is odd. In this case some extra
`teeth' like structures of length $1$ are added near the boundary in the
$\epsilon^2$ direction.
\end{proof}

\begin{figure}
	\includegraphics[width=1\textwidth]{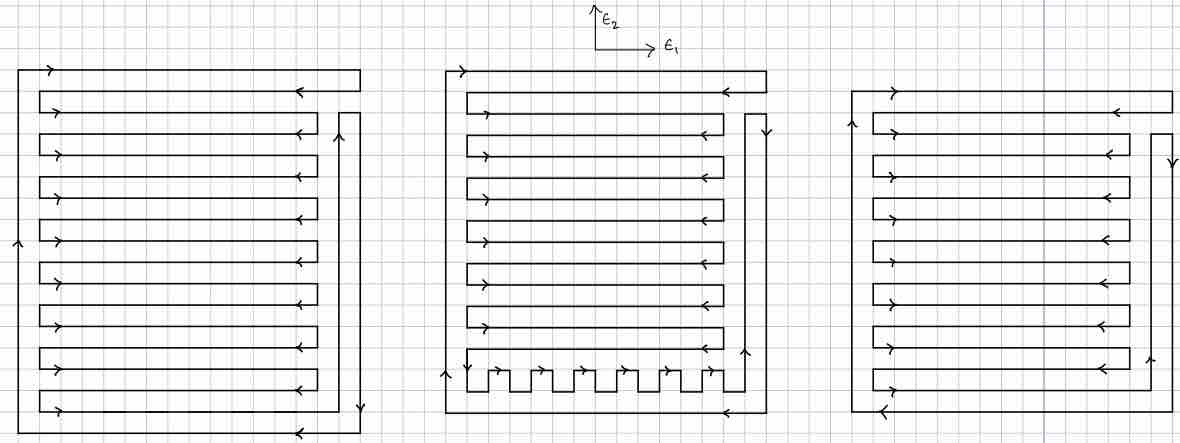}
	\caption{Hamiltonian cycles $c_{l,m}$ on a $17\times 18$, $16\times 17$
and $16\times 16$ box respectively. Notice that if we choose a vertex at
distance greater than $4$ from the boundary then the Hamiltonian cycle will just
consist of lines in the $\epsilon^1$ direction and on the boundary the cycle moves in the clockwise direction. \label{figure:Hamiltonian_1}}
\end{figure}
\subsection{Hamiltonian cycles in two dimensions}
With Lemma \ref{lemma: Hamiltonian on layer} we can readily prove property $F$ for $d=2$. By Corollary \ref{corollary: Hamiltonian tiling}, we can tile $C\setminus \cup_i C_i$ by almost $k$-boxes such that

\begin{enumerate}
	\item For each $1\leq i\leq r$, the tiling of $C\setminus \cup_{i}C_i$ meets the tiling of $C_i$ by $k$ boxes, face to face on $\partial_1(C_i)$.
	\item The tiling of $C\setminus \cup_{i}C_i$ agrees with the $k$-box tiling of $C$ on $\partial C$.
	\item The tiling of $C\setminus \cup_{i}C_i$ is almost face to face.
\end{enumerate}	

We start by constructing a Hamiltonian path on $C \setminus \cup_i C_i$.  Recall
the graph structure on almost $k$-tilings. We regard the boxes in the tiling as
vertices of a graph where we place an edge between two boxes if they
meet in some face which is at least $1/12^d$ of the minimum volume of over all
faces of the two boxes.

By using property (3) above, we fix an ordering on the almost $k$-boxes in the tiling
of $C \setminus \cup_i C_i$ such that each box (after the first) is almost face
to face with some earlier box.  We start the construction by placing a Hamiltonian
cycle $c_{l,m}$ for appropriate $l$ and $m$ on each tile in the almost $k$-box tiling $C
\setminus \cup_i C_i$.  Using the ordering above and Proposition
\ref{proposition: Halmiltonian join}, it is straightforward to iteratively
connect the Hamiltonian cycles on each almost $k$-box into one large Hamiltonian
cycle.  Note that we use that each $c_{l,m}$ is essentially a clockwise cycle
on the boundary.

By property (1) above the tiling of $C\setminus \cup_{i}C_i$ meets the tiling of
$C_i$ by $k$ boxes face to face on $\partial_1(C_i)$.  Hence we can connect each
$c_i$ to the Hamiltonian cycle on $C\setminus \cup_i C_i$.  Finally since the
tiling of $C\setminus \cup_{i}C_i$ agrees with the $k$-box tiling of $C$ on
$\partial C$ (property (2)), we can open an edge along $\partial^+_1(C)$ to get
the required Hamiltonian path with the prescribed entry and exit points.


\subsection{Hamiltonian cycles in dimensions greater than two}
For $d>2$, the situation is much more complicated. The main complication is an issue
of parity. Any box can be divided into several two dimensional layers in the
direction $\epsilon^1\times \epsilon^2$. However we cannot place the Hamiltonian cycles $c_{l,m}$ in each of these layers as it is; we have to alternate the orientation between clockwise and anticlockwise otherwise it will be difficult to connect them (using Proposition \ref{proposition: Halmiltonian join}). Keeping this in mind we define Hamiltonian cycles $d_{l,m}$ where the edges are reversed. Specifically $(\alpha, \beta)$ is a (directed) edge for $c_{l,m}$ if and only if $(\beta, \alpha)$ is an edge for $d_{l,m}$.

It is easy to see that the Hamiltonian cycles $d_{l,m}$ have analogous properties
to the ones in Lemma \ref{lemma: Hamiltonian on layer}.  In particular clockwise
is replaced with counterclockwise in (1) and the roles of $\epsilon^1$ and
$-\epsilon^1$ are reversed in (2).


Let $R = \prod_t[l_t, m_t]$ be a box such that $(m_1-l_1)(m_2-l_2)$ is even. The box
can be partitioned into two dimensional layers as follows.
\[\bigsqcup_{\alpha\in \prod_{t>2}[l_t, m_t]}[l_1, m_1]\times [l_2,
m_2]\times\{\alpha\}\]
If $\alpha$ is an even vector we place a copy of $c_{m_1-l_1, m_2-l_2}$ on
$[l_1, m_1]\times [l_2, m_2]\times\{\alpha\}$ and a copy of $d_{m_1-l_1,
m_2-l_2}$ if $\alpha$ is an odd vector. Clearly if $\alpha, \beta$ differ by a
unit vector then the corresponding Hamiltonian cycles have opposite orientation
and can be connected to each other by opening an edge in the $\epsilon^1$
direction at distance greater than $4$ from the boundary using Proposition
\ref{proposition: Halmiltonian join}.  There is a natural graph structure on
$\prod_{t>2}[l_t, m_t]$ which connects vectors differing by a unit vector. By
finding a spanning tree for this graph and connecting Hamiltonian cycles on
adjacent two dimensional layers we obtain a Hamiltonian cycle on the entire box
$R$.  This Hamiltonian cycle will be called the \emph{even type Hamiltonian
cycle} on the box. The \emph{odd type} Hamiltonian cycle will be the one
obtained by reversing all the edges.

The set $C\setminus \cup_{i}C_i$ can be tiled by almost $k$-boxes such that they
satisfy the conclusion of Corollary \ref{corollary: Hamiltonian tiling}.  As in
the two dimensional case, the graph of almost $k$-boxes (where two such boxes are considered adjacent if they meet almost face to face) is connected.  As before we order them so that each almost $k$-box in the tiling meets some earlier almost $k$-box almost face to face.  Now we would like to
iteratively connect Hamiltonian cycles on each of the almost $k$-boxes.  This works
as before with the change that at each step we need to choose either an even
type or an odd type Hamiltonian cycle so that we can connect them in accordance with Proposition
\ref{proposition: Halmiltonian join}.

Suppose that $R$ meets almost face to face with $R'$ using the faces
$\partial_t(R)$ and $\partial_t(R')$.  For the moment, consider the unique
undirected cycles on $R$ and $R'$.  If $t=1$, then the undirected cycles have
parallel edges in the $\pm\epsilon^2$ direction on $\partial_t(R)$ and
$\partial_t(R')$.  If $t \geq 2$, then the undirected cycles have parallel edges
in the $\pm\epsilon^1$ direction on $\partial_t(R)$ and $\partial_t(R')$.  In
either case, for each choice of even or odd type cycle on $R$, there will be a
choice of even or odd type cycle on $R'$ which satisfies the hypothesis of
Proposition \ref{proposition: Halmiltonian join}. As before we open an edge on $\partial_1^+ C$ as required by the definition of
$P_C$.

The remaining issue is to connect the Hamiltonian path $\tilde{c}$ constructed on
$C \setminus \cup_i C_i$ with the existing paths $c_i$ on each $C_i$.  By
property (1) of Corollary \ref{corollary: Hamiltonian tiling}, the tiling of $C
\setminus \cup_i C_i$ is face to face with the tiling of $C_i$ on
$\partial_1(C_i)$.  Let $\alpha$ be the starting vertex of $c_i$ and $\beta$ be
the ending vertex.  From the definition of $P_{C_i}$ and the definition of the
Hamiltonian cycles on almost $k$-boxes above, $\tilde{c}$ has an edge parallel to
$\{\alpha,\beta\}$, but the direction of the edge may not allow us to apply
Proposition \ref{proposition: Halmiltonian join} to connect them.  Suppose that
$R$ is the almost $k$-box in the tiling of $C \setminus \cup_i C_i$ which is
adjacent to $\alpha$ and $\beta$. As stated in Corollary \ref{corollary: Hamiltonian tiling}, $R$ is a translate of the box $[1, l]\times [1,k]^{d-1}$ for some $l\geq k$.

We revise the definition of $\tilde{c}$ on $R$ to reverse the edge parallel to
$\{\alpha,\beta\}$.  For simplicity we assume that the Hamiltonian cycle we used on
$R$ is of even type.  Partition $R$ into two pieces $\partial_1^-(R)$ and $R
\setminus \partial^-_1(R)$.  Note that we can place a Hamiltonian cycle $e$ of even
type on $R \setminus \partial^-_1(R)$.  This ensures that all the previous
connections made through $R$ in $\tilde{c}$ can still be made with the new
cycle.  Next we construct a Hamiltonian cycle on $\partial^-_1(R)$ which can be
connected to both $c_i$ and $e$.



Recall that $\partial_1^{-}(R)$ is a translate of $\{1\}\times [1,k]^{d-1}$.  We
partition $\partial_1^{-}(R)$ into layers parallel to $\epsilon^2\times
\epsilon^3$.  Let $L$ be the layer adjacent to $\alpha$ and $\beta$.  We will place a
Hamiltonian cycle on $L$ which can be connected to both $c_i$ and $e$. First rotate $c_{k,k}$ and $d_{k,k}$ to the $\epsilon^2\times \epsilon^3$ plane so that edges at distance more than $4$ from the boundary are either in the $\epsilon^2$ direction or in the $-\epsilon^2$ direction. Choose one among these orientations so that it connects with $e$ and finally modify it (as in Figure \ref{figure:hamiltonian_6}) to be able to connect it with $c_i$.


\begin{figure}
	\includegraphics[width=.5\textwidth]{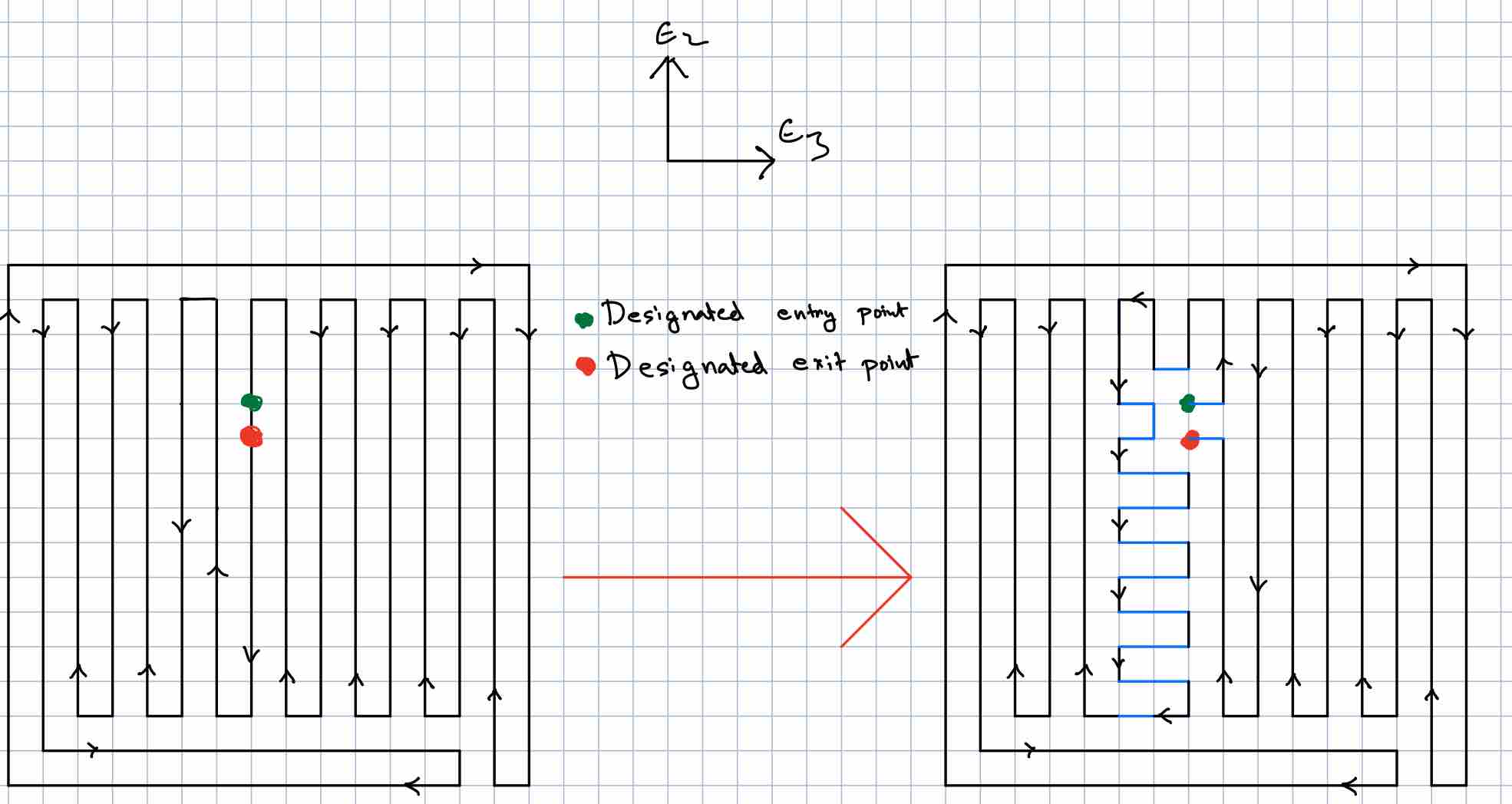}
	\caption{\label{figure:hamiltonian_6}We see that original choice of the Hamiltonian cycle on the left do not allow the designated entry and exit points. We first perturb the Hamiltonian cycle to ensure the entry and exit point. Keeping the upper part of the Hamiltonian cycle exactly the same we introduce teeth below the perturbed region to reach all the vertices below.}
\end{figure}

To finish we need to extend to the rest of the layers of $\partial^-_1(R)$.  As
in previous arguments, we order the layers starting with $L$ such that any layer
has distance $1$ to some previous layer (recall that $\partial^-_1(R)$ is just a translate of $\{1\}\times [1,k]^{d-1}$).  Then we iteratively extend our
Hamiltonian cycle on $L$ to the other layers using the rotated versions of $c_{k,k}$ and $d_{k,k}$ (given above) alternatively.
This finishes the modification of the Hamiltonian cycle on $R$ and with it the
proof that the space of Hamiltonian cycles satisfies extension portion of property
F.

\subsection{The compactness criterion}
Suppose that $x$ is a configuration such that for an increasing union of
$100k$-grid unions $B_i$ we have that $x|_{B_{i}}\in P_{B_{i}}$ for all $i$. To complete the proof of Property $F$ we need to show that $x$ is bi-infinite directed Hamiltonian path. For this it is enough to prove that for all $\alpha, \beta\in \Z^d$ there is a directed path along $x$ which goes from $\alpha$ to $\beta$ or vice versa and that every $\gamma\in \Z^d$ has a successor and the predecessor. For this, take $i$ large enough such that $\alpha, \beta\in B_i$. Then immediately we have that there is a directed path along $x$ of length less than or equal to $|B_i|$ connecting the two. Finally take $i'$ large enough such that $\gamma$ and its neighbours are in  $B_{i'}$. Then we have by construction that $\gamma$ must have both a predecessor and a successor in $B_{i'}$.

This completes the proof.

\section{Proper edge colourings have property $F$}\label{section: proper edge colouringe}
Recall that $E_t$ denotes the set of injective maps from the set of unit vectors to the set of colours $\{1,2, \ldots, t\}$ and $X^{(t)}$ denotes the space of proper $t$-edge colourings of $\Z^d$. We can think of the set of proper $t$-colourings both as maps from edges of $\Z^d$ to the set of colours and as maps from $\Z^d$ to $E_t$. We will use both these (equivalent) perspectives interchangeably.

\begin{thm}\label{theorem: edge colourings have property F} Let $t\geq 2d$. The space of proper $t$-colourings $X^{(t)}$ 
	has property $F$ for a sequence of patterns with entropy equal to the
	topological entropy of the $X^{(t)}$.  \end{thm}

\begin{remark}
This shows that Borel edge chromatic number of any free Borel $\Z^d$ action is $2d$.
\end{remark}

\begin{proof}Recall that given a set $C\subset \zd$ we defined $\partial_i^+(C)$ as the set of $\alpha\in C$ such that $\alpha+\epsilon^i \notin C$ while $\partial_i^-(C)$ was defined as  the set of $\alpha\in C$ such that $\alpha-\epsilon^i \notin C$.
Fix $t\geq 2d$. Let $C\subset \Z^d$ be connected. For $d<i\leq 2d$ let $\epsilon^i=-\epsilon^{i-d}$. Let 
\begin{eqnarray*}P_C=\left\{a\in \mathcal L(X^{(t)}, C)~:~ \text{ the outward edges normal to the faces $\partial_i^+(C) $ and $\partial_i^-(C)$ are coloured $i$ in $a$}\right\}.
	\end{eqnarray*}
This boundary condition is amenable to certain natural reflections which we will describe next.

 Let $a$ be a proper $t$-edge colouring of a box $B$ and let $F=\partial_i^+(B)$. Then we can reflect $a$ along $F+1/2 \epsilon^i$ to get a proper $t$-edge colouring of the union of the box $B$ and its reflection along $F$. This symmetry is remarkably useful and very similar to the symmetry available to us from graph homomorphisms. In particular it allows, as in \cite[Section 9]{MR4311117}, application of ``reflection positivity'' which gives us that the entropy of these patterns is equal to the entropy of the subshift. Since it is essentially the same argument, we do not repeat it here. We are left to prove property $F$ for these patterns. For this we will use Proposition \ref{proposition: tiling of boxes} in a crucial manner.

 We fix $N,k=10$. Let $C, C_1, C_2, \ldots, C_r$  be $Nk$-grid unions with offsets $\m 0, \alpha^1, \alpha^2, \ldots, \alpha^r\in [1,nk]^d$ such that  $C_1+[-100,100]^d$,  $C_2+[-100,100]^d$, \ldots,  $C_r+[-100,100]^d$ are disjoint and contained in $C$. By Proposition \ref{proposition: tiling of boxes}, we know that $$C\setminus \cup_{i=1}^r C_i$$ can be tiled by almost $100$-boxes, that is, boxes all of whose side lengths are at least $100$ and at most one of them is different from $100$. Thus to complete the proof we need to show that $P_B$ is non-empty for all almost $100$-boxes $B$. If all the sides of $B$ are even then the construction is easy: The edges parallel to the direction $\epsilon^i$ are coloured $i$ and $i+d$ alternatively. This would ensure that the outward edges normal to $\partial_i^+(B)$ and $\partial_i^-(B)$ are coloured $i$. 
 
 Now suppose (without the loss of generality) that the side in the direction $\epsilon^1$ is odd and the rest are even. The edges in the directions $\epsilon^i$ are coloured $i$ and $i+d$ alternately for $i\geq 3$. The edges in the direction $\epsilon^1$ are coloured (from smaller to larger first coordinate value) 
 $1$, $2+d$, $1+d$ followed by $1$ and $1+d$ alternatively. Order the lines in the $\epsilon^2$ direction according the value of the first coordinate on the line. 
 The ones with the smallest value of the first coordinate are coloured $2$ and
$1+d$ alternatively. The lines with the next highest value of the first
coordinate are coloured $2$ and $1$ alternatively. The rest of the lines are
labelled $2$ and $2+d$ alternatively. It completes the required colouring and hence the
proof. \end{proof}


\section{Discussions and Open questions}

\subsection{From results in ergodic theory to the Borel setting} \label{subsection: open ergodic to borel}  We expand here on a comment and question by Anton Bernsthyn.  It is a natural question to ask if results regarding embeddings and factorings that have been proved in the presence of an ergodic measure can be automatically upgraded to the Borel setting. In other words, in ergodic theory the following situation often arises. Given the action of the group $\Z^d$ on a standard probability space $(X, \mu)$, there is a map from $X$ to a finite set $\{1,2, \ldots, l\}$ such that the image of the equivariant induced map to $\{1, 2, \ldots, l\}^{\Z^d}$ satisfies certain local constraints $\mu$ almost everywhere. The question is whether a probability measure is really required in this context. If we take the alphabet to be infinite then the answer is no due to Theorem \ref{theorem: no tiling theorem} (look at Remark \ref{remark: tilings untiling}).

\subsection{From the Borel to the continuous setting} It will be exciting to know more about the continuous analogue of our results. We mentioned many related results in the introduction (Section \ref{section: Introduction}) and direct the reader there for the references. Here we will focus on the main questions and subtleties in the subject.

It is clear that our techniques do not automatically carry forward to the continuous settings. Theorem 2.10 of \cite{gao2018continuous} showed that a continuous (clopen) version of Theorem \ref{theorem: precake} fails for the free part of the full shift $2^{\Z^d}$. 

Since Theorem \ref{theorem: precake} is the mechanism for constructing the Borel
maps in our paper, Theorem 2.10 of \cite{gao2018continuous} shows that there is
no straightforward modification of our methods to give continuous maps.  In the
case of colorings, Corollary 4.4 of \cite{gao2018continuous} is equivalent to
the statement that there is no continuous factor map from the free part of
$2^{\Z^2}$ to the space of proper $3$-colorings.  So our result that there is a
Borel such map cannot be improved to a continuous one. Further interesting
conditions were found in \cite{gao2018continuous} (given by the twelve tiles
theorem) which give equivalent formulations of whether the free part of
$2^{\Z^2}$ can be mapped continuously and equivariantly to a given shift of
finite type.  These equivalent formulations may be difficult to check.  For
instance, it would be nice to have an answer to the following question: When can
the free part of the full shift be tiled continuously by a given set of
rectangles? From results by Gao-Jackson-Krohne-Seward we know that there is no
continuous perfect matching (domino tiling) of the free part of the full shift.
This and the fact that there is no continuous $3$-colouring are related to the
fact that perfect matchings have a non-trivial cohomology while the full shift
doesn't \cite{schmidt1995cohomology})  On the other hand, for $k>0$ there are
continuous tilings by rectangles of the form $n \times n$, $n \times (n+1)$,
$(n+1) \times n$ and $(n+1) \times (n+1)$.

Problems about continuous embeddings appear to be subtle. Results by Lightwood
\cite{MR1972240,MR2085910} say that the strong mixing condition are enough when
we talk about embeddings of subshifts without periodic points. On the other hand
results by Salo \cite{salo2021conjugacy} indicate that embeddings of the free
part of subshifts can be naturally extended to an embedding of the subshifts in
certain cases. This is very problematic because little is known about continuous embeddings between $\Z^d$-subshifts for $d\geq 2$. For $d=1$, a necessary and sufficient condition for finding embeddings into a shift of finite type is given by Krieger's embedding theorem \cite{krieger1982subsystems}.

\subsection{Full universality} One of our main results says that if $Y$ is a subshift with a collection $\mathcal C$ which has property $F$  and entropy $h(\mathcal C)$ then the free part of any subshift with entropy lower than $h(\mathcal C)$ can be Borel embedded in it (Theorem \ref{theorem: property F embedding subshift}). On the other hand, results in \cite{MR4311117} give Borel embeddings of any system modulo a universally null set into systems with a property similar to property $F$ called flexibility. 

The Gurevic entropy of a Borel $\Z^d$ system $(X, T)$ is given by the supremum of the Kolmogorov-Sinai entropy of the invariant probability measures on the space. 

We conjecture that if $(Y, S)$ is a subshift which has a collection $\mathcal C$ with property $F$ and entropy $h(\mathcal C)$ and $(X,T)$ is any free Borel $\Z^d$ system whose Gurevic entropy is strictly smaller that $h(\mathcal C)$ then  $(X, T)$ can be embedded into $(Y, S)$. For mixing shifts of finite type in one dimension (which are automatically strongly irreducible subshifts) this is known due to Hochman \cite{MR3077948,MR3880210}. Along with some arguments by Tom Meyerovitch, this conjecture follows from an upcoming result by Mike Hochman and Brandon Seward.

\subsection{Property $F$ for the non-symbolic setting} In \cite{MR4311117} a property called flexibility was defined for $\Z^d $ topological dynamical systems and it was shown (modulo a universally null set) that under appropriate entropy constraints any Borel $\Z^d$ system can be embedded into it. We could not come with an analogue of flexibility for general $\Z^d$ topological dynamical systems. In particular we would like to know whether the following is true. Let $(X, T)$ be a $\Z$ topological dynamical systems with a strong mixing property like non-uniform specification (look at \cite{MR4311117} for the definition). Does there exist an equivariant map from the free part of any Borel $\Z$ system to the free part of $(X, T)$? By \cite{MR4311117} the answer is yes modulo sets of universal measure zero.

\bibliographystyle{alpha}
\bibliography{Borel}

\end{document}